\documentclass[11pt]{article}
\usepackage[letterpaper, margin=1in]{geometry}
\usepackage{graphicx}
\usepackage{amsthm,amsmath,amssymb}
\usepackage{mathabx}
\usepackage{geometry}
\usepackage{xspace}
\usepackage{xcolor}
\usepackage{thmtools}
\usepackage{algorithm}
\usepackage[noend]{algpseudocode}
\usepackage{thm-restate}
\usepackage{hyperref}
\usepackage{cleveref}
\usepackage[T1]{fontenc}
\usepackage[utf8]{inputenc} 
\usepackage{lmodern}

\newtheorem{theorem}{Theorem}[section]
\newtheorem{corollary}[theorem]{Corollary}
\newtheorem{lemma}[theorem]{Lemma}
\newtheorem{proposition}[theorem]{Proposition}
\newtheorem{observation}[theorem]{Obervation}
\newtheorem{claim}{Claim}[theorem]

\newcommand*{\claimproofs}{Proof of the Claim.}
\newenvironment{claimproof}[1][\claimproofs]{\begin{proof}[#1]}{\end{proof}}

\definecolor{dullblue}{HTML}{6496C8}
\newcommand{\f}{\mathcal{F}}
\newcommand{\cf}{\mathcal{C}_{\mathcal{F}}}
\newcommand{\q}{\mathcal{Q}}
\newcommand{\p}{\mathcal{P}}

\newcommand{\tw}{tw\xspace}

\newcommand{\fcov}{\mathrm{fcov}_{\mathcal{F}}}

\title{Induced Minors and Coarse Tree Decompositions}
\author{Maria Chudnovsky\thanks{Princeton University, Princeton, NJ, USA. Supported by NSF Grants DMS-2348219 and CCF-2505100,  AFOSR grant FA9550-25-1-0275, and a Guggenheim Fellowship.} \and
Julien Codsi\thanks{Princeton University, Princeton, NJ, USA. Supported by NSF Grant DMS-2348219 and by the Fonds de recherche du Québec through the doctoral research scholarship 321124.} \and
Ajaykrishnan E~S\thanks{Department of Computer Science, University of California Santa Barbara, Santa Barbara, CA, USA. Supported by NSF Grant CCF-2505099.} \and
Daniel Lokshtanov$^{\ddagger}$
}
\date{}

\begin{document}

\maketitle

\begin{abstract}
Let $G$ be a graph, $S \subseteq V(G)$ be a vertex set in $G$ and $r$ be a positive integer. The {\em distance $r$-independence number} of $S$ is the size of the largest subset $I \subseteq S$ such that no pair $u$, $v$ of vertices in $I$ has a path on at most $r$ edges between them in $G$.
It has been conjectured [Chudnovsky et al., arXiv, 2025]
% \cite{chudnovsky2025treewidth}~
that for every positive integer $t$ there exist positive integers $c$, $d$ such that every graph $G$ that excludes both the complete bipartite graph $K_{t,t}$ and the grid $\boxplus_t$ as an induced minor has a tree decomposition in which every bag has (distance $1$) independence number at most $c(\log n)^d$. 
We prove a weaker version of this conjecture where every bag of the tree decomposition has distance $16(\log n + 1)$-independence number at most $c(\log n)^d$.
Along the way, we also prove a version of the conjecture where every bag of the decomposition has distance $8$-independence number at most $2^{c (\log n)^{1-(1/d)}}$.
\end{abstract}

\section{Introduction}
All graphs in this paper are finite and all logarithms are base $2$.
For standard graph theory terminology that is not defined here we refer the reader to Diestel's textbook~\cite{diestel2025graph}.
Let $G = (V(G),E(G))$ be a graph. For a set $X \subseteq V(G),$ we denote by $G[X]$ the subgraph of $G$ induced by $X$, and by $G - X$ the subgraph of $G$ induced by $V(G) \setminus X$. 
%In this paper, we use induced subgraphs and their vertex sets interchangeably. 
%
%For subsets $X,Y \subseteq V(G)$ we say that $X$ is {\em anticomplete} to $Y$ if $X$ and $Y$ are disjoint and every vertex of $X$ is non-adjacent to every vertex of $Y$.
For graphs $G$ and $H$, we say that $H$ is a {\em minor} of $G$, or that $G$ {\em contains} $H$ as a minor, if there exist disjoint connected induced subgraphs $\{X_v\}_{v \in V(H)}$ of $G$ 
such that for every pair of distinct vertices $u$, $v$ of $H$, if $uv$ is an edge of $H$ then there exists an edge from $X_u$ to $X_v$ in $G$. 
If, for every pair of distinct vertices $u$, $v$ of $H$, $uv$ is an edge of $H$ if and only if there exists an edge from $X_u$ to $X_v$ in $G$, then $H$ is an {\em induced minor} of $G$.

%In this case we say that $G$ {\em contains} $H$  an $H$-induced-minor}.
%if and only if $uv$ is an edge in $H$.
%s anticomplete to $X_v$ if and only if $u$ is non-adjacent to $v$ in $H$; in this case we say that $G$ {\em contains an $H$-induced-minor}.
Given a family $\mathcal{H}$ of graphs, we say that a graph $G$ is {\em $\mathcal{H}$-induced-minor-free} if no induced minor of $G$ is isomorphic to a member of $\mathcal{H}$.
The {\em distance} between two vertices $u$ and $v$ in $G$ is the minimum number of edges of a $u$-$v$ path in $G$. An {\em independent} set in $G$ is a set $I$ of vertices of pairwise distance at least two. For a positive integer $r$, a {\em distance} $r$-{\em independent set} is a set $I$ of vertices of pairwise distance strictly greater than $r$. Thus, a set is independent if and only if it is $1$-independent.
The {\em distance} $r$-{\em independence number} of a vertex set $S$ is denoted by $\alpha_r(S)$ and defined as the size of the largest distance $r$-independent subset of $S$. The {\em independence number} of a set $S$ is $\alpha(S) = \alpha_1(S)$.

For a graph $G$, a \emph{tree decomposition} $(T, \chi)$ of $G$ consists of a tree $T$ and a map $\chi\colon V(T) \to 2^{V(G)}$ with the following properties:
\begin{enumerate}
\itemsep -.2em
    \item For every $v \in V(G)$, there exists $t \in V(T)$ such that $v \in \chi(t)$.

    \item For every $v_1v_2 \in E(G)$, there exists $t \in V(T)$ such that $v_1, v_2 \in \chi(t)$.

    \item For every $v \in V(G)$, the subgraph of $T$ induced by $\{t \in V(T) \mid v \in \chi(t)\}$ is connected.
\end{enumerate}

For each $t\in V(T)$, we refer to $\chi(t)$ as a \textit{bag of} $(T, \chi)$.  The \emph{width} of a tree decomposition $(T, \chi)$, denoted by $width(T, \chi)$, is $\max_{t \in V(T)} |\chi(t)|-1$. 
The {\em independence} number of a tree decomposition $(T, \chi)$ is the maximum independence number $\alpha(\chi(t))$ over all bags $\chi(t)$ for $t \in V(T)$.
The {\em distance} $r$-{\em independence} number of a tree decomposition $(T, \chi)$ is the maximum distance $r$-independence number $\alpha_r(\chi(t))$ over all bags $\chi(t)$ for $t \in V(T)$.
The \emph{treewidth}, {\em tree independence} and {\em distance} $r$-{\em tree independence} number of $G$, are
the minimum width, independence number, and distance $r$-independence number of a tree decomposition of $G$, respectively. 
We will say that a {\em class} of graphs has a bounded graph parameter (such as treewidth or tree-independence) if there exists a constant $c$ such that the parameter is upper bounded by $c$ for every graph in the class. 

Treewidth has been extensively studied, as graphs whose treewidth is upper bounded by a constant have favorable structural~\cite{bodlaender1998partial, GMV} and algorithmic~\cite{cygan2015parameterized} properties. Some of the applications (such as polynomial time solvability of a number of graph problems) carry over to graphs of bounded tree-independence~\cite{dallard2024treewidth,DBLP:conf/esa/LimaMMORS24}. 
This motivates a systematic study of which classes have bounded treewidth or tree-independence.
%~\cite{abrishami2024tree,dallard2021treewidth,dallard2024treewidth}.\todo[inline]{DL: Is this enough? Maria: any papers in the treewidth and induced subgraph series we should highlight here?\\ Maria: That sentence is very broad; there are thousands of papers on treewidth.... Maybe just limit it to tree alpha and cite all the tree alpha series?  }

In terms of minors, treewidth is well understood: there exists a fixed constant $c$ such that, if  $t$ is the largest integer such that $G$ contains the $t$ times $t$ grid $\boxplus_{t}$ as a minor, then the treewidth of $G$ is at least $t$~\cite{diestel2025graph} and at most $t^9(\log t)^c$ for some fixed constant $c$~\cite{chekuri2016polynomial,chuzhoy2021towards}. 
On the other hand, recent results suggest~\cite{sintiari2021theta} (see also~\cite{DBLP:journals/corr/abs-2502-14775}) that there does not exist a ``nice'' characterization of bounded treewidth graphs in terms of forbidden induced subgraphs. For example, for every sufficiently large $n$ there exists an $n$-vertex graph which excludes a triangle $K_3$ as an induced subgraph, $K_{3,3}$ and $\boxplus_{5}$ as an induced minor (indeed even smaller graphs), and yet has treewidth (and also tree independence number) at least $c \cdot \log n$ for a constant $c > 0$~\cite{sintiari2021theta}.
%\todo[inline]{DL: i changed the paragraph above because the flat wall theorem implies that every graph of large treewidth either has a large clique or a large wall as an induced minor.}

%
Quite often, problems that are solvable in polynomial time when the treewidth or tree-independence number of the input graph is a constant can be solved in {\em quasi-polynomial} ($O(n^{(\log n)^d})$ for some constant $d$) time when the treewidth or tree-independence number of the input graph is upper bounded by a poly-logarithmic (i.e. $O((\log n)^c)$ for some constant $c$) function in the number of vertices \cite{lokshtanov2026findinglargesparseinduced}. 
Together with the above-mentioned difficulties of characterizing the graphs of bounded treewidth and tree independence in terms of induced minors, this motivates trying to characterize the graphs whose treewidth (or tree independence) is bounded poly-logarithmically in the number of vertices (see for example~\cite{tw3,treeAlpha4}). 

It was conjectured \cite{treeAlpha5} that, for every integer $t$ there exist integers $c, d$ such that every graph $G$ either contains $K_{t,t}$ or $\boxplus_t$ as an induced minor or has tree-independence at most $(c \log n)^{d}$.
A subset of the authors recently showed~\cite{chudnovsky2025treewidth} that characterizing the graphs with poly-logarithmic tree-independence is essentially equivalent to characterizing the graphs with poly-logarithmic treewidth and no clique on $c \cdot \log n$ vertices for a sufficiently large constant $c$. 

The main results of this paper are the following three theorems.
\begin{restatable}{theorem}{CoarseTreewidthOurClass}\label{thm:coarse_treewidth_our_class}
    Let $t$ be a positive integer and let $G$ be a $K_{t,t}$-induced-minor-free graph. Then there exists constants $c(t)$, $d(t)$ such that $G$ either has $\, \boxplus_t$ as an induced minor or has a tree decomposition $(T,\chi)$ such that for every $x\in V(T)$, the distance $16\log 4n$-independence number of the bag $\chi(x)$ is at most $115000 \cdot \log^3 4n \cdot f \log 4f$, where $f = (1000000\cdot d\cdot \log^5 4n)^{c}$.
\end{restatable}

\begin{restatable}{theorem}{CoarseSubpolyTreewidthOurClass}\label{thm:coarse_subpoly_treewidth_our_class}
    Let $t$ be a positive integer and let $G$ be a graph that excludes $K_{t,t}$ and $\boxplus_t$ as induced minors. Then there exists a positive integer $c(t)$ and $\epsilon(t)\in(0,1]$, such that $G$ has a tree decomposition $(T,\chi)$ such that for every $x\in V(T)$, the distance $8$-independence number of $\chi(x)$ is at most $2^{c(t) \log^{1-\epsilon(t)} n}$.
\end{restatable}

%Theorems~\ref{thm:coarse_treewidth_our_class} %and~\ref{thm:coarse_subpoly_treewidth_our_class} 

Our results fit into a recent line of work concerning ``coarse graph theory''~\cite{Agelos, nguyen2025asymptoticstructureicoarse, gartland2023inducedversionsmengerstheorem, Hendrey_2024, Albrechtsen_2024, nguyen2025counterexamplecoarsemengerconjecture}.
In this field tree decompositions where each bag can be covered by 
``not too many'' balls of ``small'' radius are of interest. Another line of inquiry within coarse graph theory is to understand what versions of the classic Menger's Theorem~\cite{Menger} hold true when disjoint paths are replaced by paths that are sufficiently far apart, and bounded size separators are replaced by separators that can be covered by ``not too many'' balls of ``small'' radius.  Note that by
Lemma~6.2 of \cite{subpolynomialtreewidth}, the existence of  a tree decomposition as above implies the Menger theorem analogue. 
Theorems~\ref{thm:coarse_treewidth_our_class} and~\ref{thm:coarse_induced_menger_our_class} fit into this framework
if ``small'' means $O(\log n)$, and ``not too many'' means $O((\log n)^c)$ for some constant $c$, while in 
Theorem~\ref{thm:coarse_subpoly_treewidth_our_class} 
``small'' is $O(1)$ and ``not too many'' is $O(2^{(\log n)^{1-\epsilon}})$ for some positive constant $\epsilon$.  
%
%Theorem~\ref{thm:coarse_induced_menger_our_class} below is the Menger's theorem analogue of Theorem~\ref{thm:coarse_treewidth_our_class}.

% 
%While Theorem~\ref{thm:coarse_induced_menger_our_class} stated below, is a ``coarse Menger's theorem'' for graphs that exclude a $K_{t,t}$ as an induced minor, which we prove along the way.

% \todo[inline]{We already say that Theorem 1.1 is a poly-log coarse version, so we should not say that Theorem 1.3 is the only one that fits into the coarse world. We should rephrase emphasizing the type of bounds on the radius} 
% \todo[inline]{(Ajay) Gave it a shot}
% \todo[inline]{(Maria) I'm happy with this now (I changed it a bit more)}
% Seconded

\begin{restatable}{theorem}{InducedMengerOurClass}\label{thm:coarse_induced_menger_our_class}
    Let $t$ be a positive integer, and let $G$ be a $K_{t,t}$-induced-minor-free graph. Let $A,B \subseteq V(G)$, and let $k$ be a positive integer. Then there exist constants $d(t)$ and $\mu(t)$ such that one of the following holds:
    \begin{enumerate}
        \item $G$ has an $A$--$B$ separator $S \subseteq V(G)$ whose distance $16\log 4n$-independence number is at most $100k \cdot g \cdot \log^3 4n$, where $g := \bigl(24\log^2 4n + 4d(t) + 1\bigr)^{\mu(t)} \cdot \bigl(2\mu(t)+1\bigr)^{4\mu(t)^2+1}$.
        \item $G$ contains a collection of $k$ vertex-disjoint $A$--$B$ paths that are pairwise anticomplete.
    \end{enumerate}
\end{restatable}

\paragraph{Proof Outline.}

% Note for the the purpose of this outline, whenever we say that a quantity is constant or bounded we mean that it is upper bounded by a fixed function of $t$.

The proofs of all main theorems begin with a reduction from the class of all $K_{t,t}$-induced-minor-free graphs to the subclass of such graphs with degeneracy bounded by a constant.
To obtain this reduction, we show that every $K_{t,t}$-induced-minor-free graph $G$ admits a partition of $V(G)$ into a constant number of parts such that each connected component of every part is contained in a ball of radius four in $G$ (this is Theorem~\ref{thm:ktt_free_implies_bounded_radius_partition}). Contracting each such component yields an induced minor $G'$ of $G$ with bounded chromatic number, and hence bounded clique number. Since $G'$ remains $K_{t,t}$-induced-minor-free, by a theorem in~\cite{MatijaPolyDegen} or~\cite{10.1093/imrn/rnaf025}, we conclude that $G'$ has bounded degeneracy.
% 
% \todo[inline]{Mention \cite{10.1093/imrn/rnaf025} as well}
% Done
% 
By construction, separators consisting of a bounded number of balls, as well as collections of vertex-disjoint and pairwise anticomplete paths in $G'$, can be lifted to corresponding structures in $G$. Therefore, for the remaining two theorems, it suffices to prove them for the subclass of $K_{t,t}$-induced-minor-free graphs whose degeneracy is bounded by a constant.

%\todo[inline]{Let us not pluralize "work", but rephrase}

A number of recent 
% \todo{this is an example of plural "works" that I objected to. There are several others}
% works
manuscripts
% \todo{(Ajay) Edited}
consider the problem of efficiently finding (or proving the existence of) separators that can be covered by a small number of sets from a prescribed family $\f$, by rounding the corresponding linear programming relaxation. In general, these linear programs may have large integrality gaps. However, positive results are known when $\f$ has additional structure -- for example when the sets are cliques~\cite{KorchemnaL0S024} or when they have poly-logarithmic independence number~\cite{chudnovsky2025treewidth}, where poly-logarithmic upper bounds on the integrality gaps can be obtained. We identify another type of $\mathcal{F}$, that we call layered families, which is amenable to the techniques used in these papers.
% \todo{(Ajay) Edited}
% previous works.

% \todo[inline]{"another type of $\mathcal{F}$" instead of "a family"}
% Done

Let $G$ be a graph and $\Pi$ be an ordered partition of $V(G)$. We define the layered family of $G$ corresponding to $\Pi$ as follows. First, remove all edges whose endpoints lie in the same part of $\Pi$. Next, orient each remaining edge from the vertex in the part with higher index to the vertex in the part with lower index. For each source, that is, each vertex of indegree zero in the resulting orientation, include in the family the set of all vertices reachable from it by a directed path. In graphs of bounded degeneracy, there exists an ordered partition $\Pi$ for which the layered family satisfies the following useful properties: (1) the integrality gap of the linear programming relaxation for separators using sets of this family is at most poly-logarithmic, (2) for every vertex $v$ in the graph, $N[v]$ can be covered by a constant number of sets from $\f$, and (3) every set in $\f$ is contained in a ball of radius $O(\log n)$. The main technical contribution of our work is to prove property (1). After this, we follow the
% framework of the previous work and round the linear programming relaxation, as follows. 
established
% \todo{(Ajay) Edited} 
framework of rounding the linear programming relaxation, as follows.
% 
% After this, we follow the framework of \cite{KorchemnaL0S024} and \cite{chudnovsky2025treewidth} and round the linear programming relaxation, as follows. 
% 
% \todo[inline]{Let us not pluralize "work", but rephrase}
% \todo[inline]{(Ajay) Unsure what you meant.}
 
For Theorem~\ref{thm:coarse_treewidth_our_class}, when the optimum objective value is ``low'', we can round the linear program and obtain balanced separators that can be covered by a poly-logarithmic number of sets from the family $\f$ and when the objective value is ``high'', we use the linear program to guide us in sampling an induced subgraph $H$ with treewidth which is large compared to its poly-logarithmic maximum degree. This forces $H$, and hence $G$, to contain $\boxplus_t$ as an induced minor, since graphs excluding both $K_{t,t}$ and $\boxplus_t$ as induced minors have treewidth polynomially bounded in their maximum degree, which in $H$ is only poly-logarithmic (and therefore small).

For Theorem~\ref{thm:coarse_subpoly_treewidth_our_class}, we apply a theorem from~\cite{subpolynomialtreewidth} which asserts that graphs that exclude $K_{t,t}$ and $\boxplus_t$ as induced minors and $K_t$ as a subgraph have sub-polynomial treewidth.
%and use Theorem~\ref{thm:ktt_free_implies_bounded_radius_partition}.

Finally, for Theorem~\ref{thm:coarse_induced_menger_our_class}, a linear program with ``low'' objective value can be rounded to obtain $A$--$B$ separators that can be covered by a poly-logarithmic number of sets from $\f$. When the value is high, we sample an induced subgraph $H$ of poly-logarithmic maximum degree containing many vertex-disjoint $A$--$B$ paths. As $H$ remains $K_{t,t}$-induced-minor-free, we can apply a theorem from~\cite{bonnet2023treewidth} to obtain an edge partition $\Pi$ in which every connected component induced by a part is a star. But now, as $H$ admits $\p$, $A$ and $B$ cannot be separated using a small number of vertices, and as maximum degree is poly-logarithmic, this property also holds for the induced minor of $H$ obtained by contracting connected components of a fixed part in $\Pi$. This observation helps us find an induced subgraph $H'$ of $H$ that preserves a reasonable fraction of the $A$--$B$ paths in $\p$, but has constant maximum degree. But, as $H'$ has constant degree and admits a large number of vertex-disjoint $A$--$B$ paths, by a theorem from~\cite{gartland2023inducedversionsmengerstheorem}, $H'$ and consequently $G$ also admit a lot of vertex-disjoint and pairwise anticomplete $A$--$B$ paths.

% \todo[inline]{Say something about Theorem 1.2}
% \todo[inline]{(Ajay) Gave it a shot}
% \todo[inline] {I like it. I only added a reference to Theorem 7.1 because I think it's a really good result too}
% Awesome

\section{Preliminaries}
When clear from context, we denote by $n$ the number of vertices of the considered graph $G$. We use $[k]$ to denote the set $\{1, 2, \dots, k\}$ for a positive integer $k$. Unless the base is specified, logarithms are in base $2$. For a set $A$ and a subset $B \subseteq A$, we write $B \subsetneq A$ to denote that $B$ is a proper subset of $A$.
A {\em clique} in a graph $G$ is a set of pairwise adjacent vertices. For a positive integer $t$, $K_t$ denotes a clique on $t$ vertices. A \emph{complete bipartite graph} is a graph $G$ with vertex set $A\cup B$ such that a pair of vertices $(u,v)$ form an edge if and only if $\{u,v\}\cap A$ contains a single vertex. For a positive integer $t$, $K_{t,t}$ denotes a complete bipartite graph where $A$ and $B$ contain exactly $t$ vertices each. For every positive integer $t$, a \emph{$t$ times $t$ grid} denoted as $\boxplus_t$ is the graph with vertex set $[t] \times [t]$ where every pair $(i, j)$, $(i', j')$ of vertices are adjacent if and only if $|i-i'| + |j-j'| = 1$.

For graphs $G$ and $H$, we say that $H$ is a {\em minor} of $G$, or that $G$ {\em contains} $H$ as a minor, if there exist disjoint connected induced subgraphs $\{X_v\}_{v \in V(H)}$ of $G$ such that for every pair of distinct vertices $u$, $v$ of $H$, if $(u,v)$ is an edge of $H$ then there exists an edge from $X_u$ to $X_v$ in $G$. If, for every pair of distinct vertices $u$, $v$ of $H$, $(u,v)$ is an edge of $H$ if and only if there exists an edge from $X_u$ to $X_v$ in $G$, then $H$ is an {\em induced minor} of $G$.

%\todo[inline]{The previous paragraph and the few lines after this are a repeat from the Intro. Is this intentional?}
%\todo[inline]{(Ajay) I wanted to have all the definitions in one place so that the reader can find it in a single place.}

Given a vertex subset $X\subseteq V(G)$ the subgraph of $G$ induced by $X$ is the graph $G[X]$ with vertex set $X$ and edge set $\{(u,v)\in E(G) \mid u,v\in X\}$. We use $G-X$ to denote the graph induced by $V(G)\setminus X$.
We use $N(v)$ to denote the set of vertices that are adjacent to $v$ and $N[v]$ for $N(v)\cup\{v\}$. Similarly, $N(X)$ denotes $\{u\in V(G)\setminus X\ |\ u\in N(v) \text{ for some } v\in X\}$ and $N[X]$ denotes $N(X)\cup X$. The \emph{degree} of $v$ is $|N(v)|$ and the \emph{maximum degree} of $G$, denoted $\Delta(G)$, is $\max_{v \in V(G)} \deg(v)$.
% 
% \todo[inline]{$u \in V(G) \setminus X$, right?}
% Yes
% 
A \emph{walk} in a graph is a sequence of vertices in which each vertex is adjacent to the next, while a \emph{path} is a walk in which every vertex is distinct. The \emph{length} of a path is equal to the number of vertices on it. Given vertex subsets $A,B\subseteq V(G)$, an \emph{$A$--$B$ walk} (resp., \emph{$A$--$B$ path}) is a walk (resp., path) starting in $A$ and ending in $B$. Likewise, for vertices $u,v\in V(G)$, a \emph{$u$--$v$ walk} or \emph{$u$--$v$ path} starts at $u$ and ends at $v$. A path $P$ in $G$ is \emph{induced} if $G[P]$ has exactly $|V(P)|-1$ edges. 

For positive integers $k$ and $r$, we say that a subset $S$ of the vertices of graph $G$ is \emph{$(k,r)$--coverable} if there exists a vertex subset $C\subseteq V(G)$ of size at most $k$ such that every vertex of $S$ has a path on at most $r$ vertices to at least one vertex of $C$. A {\em distance} $r$-{\em independent set} is a set $I$ of vertices such that no pair of vertices in $I$ has a path between them on at most $r$ vertices. 
We remark that the two notions are closely related. Suppose first that $S \subseteq V(G)$ is such that the largest distance-$r$-independent set $I \subseteq S$ in $G$ has size $k$. Then $S$ is $(k,r)$–coverable in $G$. Indeed, if there exists a vertex $v \in S$ that is not within distance $r$ of any vertex of $I$, then $I \cup \{v\}$ is again distance-$r$-independent in $G$, contradicting the maximality of $I$.
Conversely, suppose that $S$ is $(k,r)$–coverable in $G$, and let $C \subseteq V(G)$ be a set of size at most $k$ witnessing this. We claim that every distance $2r$-independent set $I \subseteq S$ has size at most $k$. Indeed, if $|I| > k$, then by the pigeonhole principle there exist distinct vertices $u,v \in I$ that are both within distance at most $r$ of the same vertex of $C$. But this contradicts the assumption that $I$ is distance $2r$-independent.

For a positive real number $\phi$, an \emph{$(X,\phi)$-balanced separator} is a vertex subset $S$ such that every connected component of $G-S$ contains at most $\phi \cdot |X|$ vertices from $X$. If $A$, $B$ are vertex subsets, then an \emph{$A$--$B$ separator} is a vertex subset $S$ such that every $A$--$B$ path in $G$ contains at least one vertex from $S$. Let $\f$ be a family of vertex subsets of $G$. A \emph{fractional $(A,B)$-separator} using $\f$ in $G$ is an assignment $\{x_F\}_{F\in\f}$ of non-negative real numbers to elements of $\f$ such that $\sum_{\substack{F\in\f\\ F\cap P\neq\emptyset}}x_F$ is at least one for every $A$--$B$ path $P$. 
A \emph{fractional cover} of $X$ using $\f$ is an assignment $\{x_F\}_{F \in \f}$ of non-negative real numbers to elements of $\f$ such that $\sum_{\substack{F \in \f \\ v \in F}} x_F \geq 1$ for every $v \in X$. The \emph{fractional cover number} of $X$ with respect to $\f$ denoted $\mathrm{fcov}_{\f}(X)$ is the minimum of $\sum_{F\in\f}x_F$ over all fractional covers $\{x_F\}_{F \in \f}$ of $X$. A subset $\mathcal{F'}$ of $\f$ is a \emph{cover} of $X$ using $\f$ if every vertex of $X$ is contained in at least one element of $\f'$. The \emph{cover number} of $X$ using $\f$, denoted by $\mathrm{cov}_{\f}(X)$ is the size of the smallest subfamily $\f'$ of $\f$ such that every vertex of $X$ is contained in some element of $\f'$. 

A \emph{Bernoulli random variable} with success probability $p$, is one which takes the value $1$ with probability $p$, and $0$ with probability $1-p$. 
We recall a version of the Chernoff bound. A stronger version of this result, along with its full proof, is available in \cite{mitzenmacher2017probability}.

\begin{proposition}[\cite{hagerup1990guided}; \cite{mitzenmacher2017probability}]\label{thm:chernoff}
    Let $X_1, X_2, \dots, X_\ell$ be independent Bernoulli random variables, and let $X = \sum_{i \in [\ell]} X_i$. If $R \geq 6\,\mathbb{E}[X]$, then
    $$
    \mathbb{P}[X \geq R] \leq 2^{-R}.
    $$
\end{proposition}

\section{Layered Family}

Let $G$ be a graph, and let $\Pi := (L_1,\dots,L_k)$ be an ordered partition of $V(G)$. We define the \emph{layer graph} $\vec{G}$ of $(G,\Pi)$ as the directed graph with $V(\vec{G}) := V(G)$ and $(u,v) \in E(\vec{G})$ if and only if $\{u,v\} \in E(G)$ and $u \in L_j$, $v \in L_i$ for some $i<j$. For vertices $u,v \in V(G)$, we say that $v$ is a \emph{descendant} of $u$, and $u$ is an \emph{ancestor} of $v$, if there exists a directed path from $u$ to $v$ in $\vec{G}$. Note that a vertex $u$ is always an ancestor of itself. Similarly, $v$ is a \emph{child} of $u$, and $u$ is a \emph{parent} of $v$, if $(u,v)$ is an arc of $\vec{G}$. 
% \todo{(Julien) make it clear that a vertex is its own ancestor\\[5pt] (Ajay) Added this line}

The \emph{layered family} $\f$ corresponding to $(G,\Pi)$ is defined as 
% follows. Initialize $\f := \emptyset$. For every vertex $u \in V(G)$ that has no parent, add to $\f$ the set $F$ that contains $u$ and all of its descendants; the vertex $u$ used to define such a set $F$ is called the \emph{center} of $F$.
% \[\left\{\{v \text{ | is a descendant of u}\}\big| \text{ u has no parent }\right\}.\] 
\[\left\{\{v\in V(G) \text{ | $v$ is a descendant of $u$}\}\big| \text{ $u\in V(G)$ has no parent }\right\}.\] The vertex $u$ used to define such a set $F$ is called the \emph{center} of $F$.
A family $\f$ of vertex subsets of $G$ is called a \emph{layered family} of $G$ if there exists an ordered partition $\Pi$ of $V(G)$ such that $\f$ is the layered family corresponding to $(G,\Pi)$. If $k$ is the smallest positive integer such that every $F\in\f$ intersects at most $k$ parts of $\Pi$, then we say that $\f$ has \emph{thickness} $k$.
% 
% \todo[inline]{Can we say "thinkness $k$" instad? $k$ layers feels too much like the number of sets in the partition}
% Done
% 
We use $\cf$ to denote the set of centers of elements of $\f$. If $u\in \cf$, then we use $F_u$ to denote the element of $\f$ that has $u$ as a center. Similarly, $\f_S$ denotes the subfamily $\{F_u \in \f\,|\, u\in S\}$, where $S\subseteq \cf$. It follows from the definition that $\f$ covers $V(G)$ and that the center of each set in $\f$ is unique to that set. The following lemma is immediate from the above definitions.

\begin{lemma}\label{lem:ancestral_path}
    Let $G$ be a graph and $k$ be a positive integer. If $\f$ is a layered family of $G$ with thickness at most $k$, then for every $F \in \f$ and every $u,v \in F$, there exists a path in $G[F]$ from $u$ to $v$ that contains at most $2k-1$ vertices, all of which are ancestors of $u$ or $v$.
\end{lemma}
\begin{proof}
    Let $\Pi$ be the ordered partition of $V(G)$ such that $\f$ is the layered family corresponding to $(G,\Pi)$. Let $F \in \f$ be a set with center $w$, and let $u \in F$. By definition, there exists a directed path from $w$ to $u$ in the layer graph of $(G,\Pi)$, contained within $F$. Let $P = (v_1,\dots,v_\ell)$ be such a path, where $v_1 = w$ and $v_\ell = u$. Each vertex $v_i$ is a child of $v_{i-1}$, and a parent of a vertex must lie in a part of $\Pi$ with a strictly larger index; hence the vertices of $P$ all lie in distinct parts of $\Pi$. Since $\f$ has thickness at most $k$, the path $P$ uses at most $k$ parts of $\Pi$, and therefore $\ell \leq k$. Thus, $P$ is a path from $w$ to $u$ with at most $k$ vertices, all of which are ancestors of $u$.
    
    Now let $u,v$ be two vertices in $F$ and let $P_u$ and $P_v$ be the paths in $G[F]$ from $u$ and $v$ to the center of $F$, of length at most $k$, consisting solely of ancestors of $u$ and $v$, respectively. These two paths imply the existence of a $u$--$v$ walk $Q$ of length at most $2k-1$ contained within the set of ancestors of $u$ or $v$. Hence, the shortest $u$--$v$ path $P$ in $G[Q]$ has length at most $2k-1$ and lies within the same set. This completes the proof of the lemma.
\end{proof}

Now, for a vertex subset $S \subseteq V(G)$, we define its \emph{upward} and \emph{downward closures} by
\begin{align*}
    \widehat{S} \; &:=\;
        \{\, u \in V(G) \mid \exists v \in S \text{ such that } u \text{ is an ancestor of } v \,\}, \\[4pt]
    \widecheck{S} \; &:=\;
        \{\, u \in V(G) \mid \exists v \in S \text{ such that } u \text{ is a descendant of } v \,\}.
\end{align*}
We say that a vertex set is \emph{upward} (resp., \emph{downward}) \emph{closed} if the upward (resp., downward) closure of the set is itself. Observe that the elements of $\f$ are downward closed by definition. We now prove some lemmas on the interactions between layered families and downward closed sets that will be useful in the next sections.

\begin{lemma}\label{lem:layered_family_minus_downward_closed}
    Let $G$ be a graph with a layered family $\f$ of thickness at most $k$. If $S\subseteq V(G)$ is a downward closed set, then $\f' := \{F \setminus S \mid F \in \f \text{ and } F\setminus S \neq \emptyset\}$ is a layered family of $G - S$ with thickness at most $k$.
\end{lemma}
\begin{proof}
    Let $\Pi := (L_1,\dots,L_k)$ be an ordered partition of $V(G)$ such that $\f$ is the layered family corresponding to $(G,\Pi)$, and set $\Pi' := (L_1\setminus S,\dots,L_k\setminus S)$. Let $\f''$ denote the layered family of $(G - S,\Pi')$; we show $\f' = \f''$, which gives both conclusions of the lemma since each $F\in \f$ intersects at most $k$ parts of $\Pi$ and consequently $F\setminus S$ intersects at most $k$ parts of $\Pi'$.
    
    We first show that $\mathcal{C}_{\f''} = \mathcal{C}_{\f'}$. Let $w \in V(G)\setminus S$. We claim that $w$ is a center of $\f''$ if and only if $w \in \cf \setminus S$. We focus on the ``if'' direction, since the ``only if'' direction is immediate: any $w \in \cf$ has no parent in the layer graph of $(G,\Pi)$, and in particular none in $(G - S,\Pi')$. Suppose then that $w \notin \cf$, so $w$ has a parent $w'$ in the layer graph of $(G,\Pi)$. If $w' \in S$, downward closure of $S$ would force the descendant $w$ into $S$, contradicting $w \in V(G) \setminus S$. Hence $w' \in V(G) \setminus S$, which implies that $w'$ is a parent of $w$ in $(G - S,\Pi')$. Combined with the observation that for $w \in \cf$, downward closure of $S$ gives $F_w \setminus S \neq \emptyset$ if and only if $w \notin S$, we obtain $\mathcal{C}_{\f''} = \cf \setminus S = \mathcal{C}_{\f'}$.
    
    Fix $w\in \mathcal{C}_{\f''}$ and let $F\in\f$ be the element with center $w$. Any directed path from $w$ in the layer graph of $(G - S,\Pi')$ is also one in the layer graph of $(G,\Pi)$, so the element of $\f''$ with center $w$ is contained in $F\setminus S$. Conversely, let $u$ be a vertex of $F\setminus S$, and consider any directed path $(w, v_1, \dots, v_{\ell-1}, u)$ from $w$ to $u$ in the layer graph of $(G,\Pi)$. We claim that every $v_i$ lies in $V(G)\setminus S$. Indeed, $u$ is a descendant of $v_i$; if $v_i\in S$, downward closure would force $u\in S$, contradicting $u\in F\setminus S$. Hence the entire path lies in $V(G)\setminus S$, and is therefore a directed path in the layer graph of $(G - S,\Pi')$, witnessing $u\in F''$. The two elements with center $w$ therefore coincide. Thus $\f' = \f''$, which completes the proof of the lemma. 
\end{proof}

\begin{lemma}\label{lem:layered_family_restricted_to_component}
    Let $G$ be a graph with a layered family $\f$ of thickness at most $k$. If $C$ is a connected component of $G$, then $\f' := \{F \cap C \mid F \in \f \text{ and } F\cap C \neq \emptyset\}$ is a layered family of $G[C]$ with thickness at most $k$.
\end{lemma}
\begin{proof}
    Let $S := V(G)\setminus C$. We claim that $S$ is downward closed. Indeed, suppose $v\in S$ and let $u$ be a descendant of $v$ in the layer graph of $(G,\Pi)$, where $\Pi$ is any ordered partition of $V(G)$ such that $\f$ is the layered family corresponding to $(G,\Pi)$. Then there exists a directed path from $v$ to $u$ whose underlying edges all lie in $G$, so $u$ and $v$ lie in the same connected component of $G$. As $v\notin C$, neither does $u$, and hence $u\in S$.
    Applying Lemma~\ref{lem:layered_family_minus_downward_closed} to $S$ now yields the result, since $G - S = G[C]$ and $F\setminus S = F\cap C$ for every $F\in\f$.
\end{proof}

\begin{lemma}\label{lem:component_is_downward_closure}
    Let $G$ be a graph with a layered family $\f$, and let $S \subseteq \cf$. Let $C$ be a connected component of $G - \widecheck{S}$, and let $Z := (\cf \cap C)$. Then $C \cup \widecheck{S} \;=\; \widecheck{S\cup Z}.$
\end{lemma}
\begin{proof}
    Let $\Pi := (L_1,\dots,L_k)$ be an ordered partition of $V(G)$ such that $\f$ is the layered family corresponding to $(G,\Pi)$. Since $\widecheck{S \cup Z} = \widecheck{S} \cup \widecheck{Z}$, it suffices to show that $\widecheck{Z} \cup \widecheck{S} = C \cup \widecheck{S}$. We first observe that for every directed path $(v_1, \dots, v_\ell)$ in the layer graph of $(G,\Pi)$ with $v_\ell \notin \widecheck{S}$, the vertex set $\{v_1, \dots, v_\ell\}$ is contained in $V(G) \setminus \widecheck{S}$, and hence in a single connected component of $G - \widecheck{S}$. Indeed, if $v_i \in \widecheck{S}$ for some $i\in [\ell]$, downward closure of $\widecheck{S}$ would force $v_\ell \in \widecheck{S}$, contradicting $v_\ell \notin \widecheck{S}$.
    
    Now, for the inclusion $\widecheck{Z} \cup \widecheck{S} \subseteq C \cup \widecheck{S}$, let $u \in \widecheck{Z} \setminus \widecheck{S}$. By definition, $u$ is a descendant of some $w \in Z$ in the layer graph of $(G,\Pi)$, and the directed path from $w$ to $u$ ends outside $\widecheck{S}$. Hence by our observation, this path lies in the connected component of $G - \widecheck{S}$ containing $w$, namely $C$, giving $u \in C$.
    For the reverse inclusion, let $u \in C$. Since $\f$ covers $V(G)$, there exists $w \in \cf$ with $u \in F_w$, i.e., $u$ is a descendant of $w$ in the layer graph of $(G,\Pi)$. As $u \notin \widecheck{S}$, the directed path from $w$ to $u$ avoids $\widecheck{S}$ and lies in the connected component of $G - \widecheck{S}$ containing $u$, namely $C$. In particular $w \in C \cap \cf$ which implies $u \in \widecheck{Z}$. This completes the proof of the lemma.
\end{proof}

\begin{lemma}\label{lem:layered_family_subfamily}
    Let $G$ be a graph with a layered family $\f$ of thickness at most $k$. If $C\subseteq \cf$, then $\f' := \{F_w \mid w\in C\}$ is a layered family of $G[\widecheck{C}]$ with thickness at most $k$.
\end{lemma}
\begin{proof}
    Let $\Pi := (L_1,\dots,L_k)$ be an ordered partition of $V(G)$ such that $\f$ is the layered family corresponding to $(G,\Pi)$, and set $\Pi' := (L_1\cap\widecheck{C},\dots,L_k\cap\widecheck{C})$. Let $\f''$ denote the layered family of $(G[\widecheck{C}],\Pi')$; we show $\f' = \f''$, which gives both conclusions of the lemma since each $F_w \in \f'$ already intersects at most $k$ parts of $\Pi$, and hence at most $k$ parts of $\Pi'$.

    Let $u \in \widecheck{C}$. We claim that $u$ is a center of $\f''$ if and only if $u \in C$. We focus on the ``if'' direction, since the ``only if'' direction is immediate: any $u \in C$ is a center of $\f$, hence has no parent in the layer graph of $(G,\Pi)$, and in particular none in $(G[\widecheck{C}],\Pi')$. Suppose then that $u \notin C$. By definition of $\widecheck{C}$, there exists $w \in C$ such that $u$ is a descendant of $w$ in the layer graph of $(G,\Pi)$. As $u \neq w$, the predecessor $u'$ of $u$ on a directed path from $w$ to $u$ is also a descendant of $w$, so $u' \in F_w \subseteq \widecheck{C}$, and $u'$ is a parent of $u$ in $(G[\widecheck{C}],\Pi')$. Hence $\mathcal{C}_{\f''} = C$.
    
    Now fix $w\in C$ and let $F''_w$ denote the element of $\f''$ with center $w$. Any descendant of $w$ in the layer graph of $(G[\widecheck{C}],\Pi')$ is a descendant in the layer graph of $(G,\Pi)$ as well, so $F''_w \subseteq F_w$. Conversely, let $u\in F_w$ and let $P$ be a directed path from $w$ to $u$ in the layer graph of $(G,\Pi)$. Every vertex on $P$ is a descendant of $w\in C$, hence lies in $\widecheck{C}$, so $P$ is a path in the layer graph of $(G[\widecheck{C}],\Pi')$, witnessing $u\in F''_w$. Therefore $F_w = F''_w$, and $\f' = \f''$.
\end{proof}

We remark that Lemmas~\ref{lem:layered_family_minus_downward_closed} and~\ref{lem:layered_family_restricted_to_component} are used in Section~\ref{sec:rounding_balanced_sep} and Lemmas~\ref{lem:component_is_downward_closure} and~\ref{lem:layered_family_subfamily} are applied in Section~\ref{sec:weird_class}. 
Finally, for vertices $u,v \in V(G)$, a $u$--$v$ path $P$ is said to be \emph{upward minimal} if there does not exist a $u$--$v$ path $P'$ such that $\widehat{P}' \subsetneq \widehat{P}$, or $\widehat{P}' = \widehat{P}$ and $|V(P')| < |V(P)|$. A path is \emph{upward non-minimal} if it is not upward minimal. We need the following property of upward minimal paths.

\begin{lemma}\label{lem:minimal_paths_are_short}
    Let $G$ be a graph and $k$ be a positive integer. If $\f$ is a layered family of $G$ with thickness at most $k$, then for every $F \in \f$ and every upward minimal path $P$ in $G$, it holds that $|F \cap P| \;\leq\; 2k - 1$. 
    % \todo{(Julien) Don't we get $2k-2$?\\[5pt] (Ajay) We should be getting the same number as Lemma~\ref{lem:ancestral_path} and there we only get $2k-1$.}
\end{lemma}
\begin{proof}
    Suppose for contradiction that there exists $F$ and $P$ such that $|F \cap P| \geq 2 k$. Let $w$ denote the center of $F$ and let $P = (v_1, \dots, v_i, \dots, v_{i'}, \dots, v_p)$, where $v_i$ and $v_{i'}$ are the first and last vertices of $F$ along $P$. Let $P_{v_i,v_{i'}}$ be the $v_i$--$v_{i'}$ path whose existence is asserted by Lemma~\ref{lem:ancestral_path}. We construct a $v_1$--$v_p$ walk $Q$ by replacing the $v_i$--$v_{i'}$ subpath of $P$ with $P_{v_i,v_{i'}}$. 
    The walk $Q$ is strictly shorter than $P$, since we removed at least $2k$ vertices and replaced them with at most $2k-1$. Furthermore, $\widehat{Q}\ =\ \widehat{P}_{v_1,v_i} \cup \widehat{P}_{v_i,v_{i'}} \cup \widehat{P}_{v_{i'},v_p}$. Note that $\widehat{P}_{v_1,v_i} \subseteq \widehat{P}$ and $\widehat{P}_{v_{i'},v_p} \subseteq \widehat{P}$ since $P_{v_1,v_i}$ and $P_{v_{i'},v_p}$ are subpaths of $P$. Moreover, $\widehat{P}_{v_i,v_{i'}} \subseteq \widehat{P}$ as well, because $P_{v_i,v_{i'}}$ consists entirely of ancestors of $v_i$ or $v_{i'}$, and is therefore contained in $\widehat{P}$. Consequently, $\widehat{Q} \subseteq \widehat{P}$. Hence, if $P'$ is the shortest induced $v_1$--$v_p$ path in $G[Q]$, then $|V(P')| < |V(P)|$ and $\widehat{P}' \subseteq \widehat{P}$, contradicting the upward minimality of $P$. Therefore, every $F \in \mathcal{F}$ and every upward minimal path $P$ satisfy $|F \cap P| \leq 2k-1$. 
\end{proof}

\section{Rounding A--B Separator LP}\label{sec:rounding_ab_sep}
Let $A$ and $B$ be vertex subsets of a graph $G$. Let $\p$ denote the set of all induced $A$--$B$ paths. The linear program corresponding to an $A$--$B$ separator using the elements of $\f$ is defined as follows, with nonnegative real variables $x_F$ for every $F \in \f$.
\begin{align}\label{lp:ab_separator_using_f}
    \text{Minimize :}&\quad \sum\limits_{F\in \f} x_F\\
    \nonumber \text{Subject To :}&\quad \sum_{\substack{F \in \f \\ F \cap P \neq \emptyset}} x_F \geq 1 && \forall P \in \p
\end{align}

The main theorem that we prove in this section is the following:

\begin{restatable}{theorem}{ABSeparatorRounding}\label{thm:ab_separator}
    Let $G$ be an $n$-vertex graph, $k$ be a positive integer and $\f$ be a layered family of $G$ with thickness at most $k$. Let $A,B \subseteq V(G)$ be vertex subsets of $G$ and suppose $\{x_F\}_{F \in \f}$ is an optimal solution to LP~\ref{lp:ab_separator_using_f} corresponding to $(G,\f,A,B)$. Then there exists an $A$--$B$ separator $S$ in $G$ such that
    $$
    \fcov(S) \;\le\; 8k  \log(2n) \cdot \sum_{F \in \f} x_F.
    $$
\end{restatable}

We begin with the following lemma.

\begin{lemma}\label{lem:x_F_behaves_nicely}
    Let $G$ be an $n$-vertex graph, $\f$ be a family of vertex subsets, $A, B \subseteq V(G)$ be vertex subsets and $\{x_F\}_{F\in\f}$ be an optimal solution to LP~\ref{lp:ab_separator_using_f}. Then, there exists another solution, $\{y_F\}_{F \in \f}$ that satisfies $1 \geq y_F \geq \frac{1}{|\f|}$ for every $F \in \f$ with $y_F > 0$, and $\sum_{F \in \f} y_F \leq 2 \cdot \sum_{F \in \f} x_F$.
\end{lemma}
\begin{proof}
    Define $\{y_F\}_{F \in \f}$ as follows,
    $$
    y_F := 
    \begin{cases}
    0 & \text{if } x_F < \frac{1}{2|\f|}, \\
    \min\{1, 2x_F\} & \text{otherwise}.
    \end{cases}
    $$
    and observe that $\sum_{F \in \f} y_F \leq 2 \cdot \sum_{F \in \f} x_F$. Furthermore, by construction, we have $1 \geq y_F \geq \frac{1}{|\f|}$ for every $F \in \f$ with $y_F > 0$. Finally, for every $A$--$B$ path $P$ we have
    $$
    \sum_{\substack{F \in \f\\ F \cap P \neq \emptyset}} y_F \quad
    \geq \quad \min\{1, \sum_{\substack{ F \in \f \\ F \cap P \neq \emptyset \\ x_F \geq \frac{1}{2|\f|}}} 2x_F \} \quad
    \geq \quad \min\left\{1, 2\left(1 - |\f| \cdot \frac{1}{2|\f|}\right)\right\}\quad
    = \quad  1,
    $$
    where the second to last transition uses that $\{x_F\}_{F\in\f}$ is a solution to LP~\ref{lp:ab_separator_using_f} and that there are at most $|\f|$ elements in $\f$ of value at most $\frac{1}{2|\f|}$.
    % \todo[inline]{(Julien) : Even if the path has at most $n$ vertices, it can intersect more than $n$ things in $\mathcal{F}$ with small weight. (as $\mathcal{F}$ is not a partition). I think that in our context, we might just add $|\mathcal{F}|\leq n$ as a condition.\\[5pt] (Ajay) Edited}
\end{proof}

\begin{lemma}\label{lem:ab_separator_with_conditions}
    Let $G$ be a graph, $k$ be a positive integer and $\f$ be a layered family of $G$ with thickness at most $k$. Let $A,B \subseteq V(G)$ be vertex subsets of $G$ and suppose $\{y_F\}_{F \in \f}$ is a solution to LP~\ref{lp:ab_separator_using_f} corresponding to $(G,\f,A,B)$ that satisfies $y_F \ge \frac{1}{n}$ whenever $y_F > 0$. Then, there exists an $A$--$B$ separator $S$ in $G$ such that $\fcov(S)$ is at most $4k \cdot \log(2n) \cdot \sum_{F \in \f} y_F$.
\end{lemma}
\begin{proof}
    For every vertex $v \in V(G)$, let $d_v = \min\,\{\, \sum_{\substack{F \in \f \\ F \cap P \neq \emptyset}} y_F \,|\, P \text{ is an $A$--$v$ path} \,\}$, let $y_v = \sum_{\substack{F \in \f \\ F \ni v}} y_F$, and let $I_v = (d_v - y_v,\ d_v]$. 
    Furthermore $d_v \leq \sum_{\substack{F \in \f \\ F \cap Q \neq \emptyset}} y_F$ for every $A$--$v$ walk $Q$, since there exists an $A$--$v$ path in $G[Q]$. 
    % 
    % Hence, if $(u,v)$ is an edge in $G$, then $d_v \leq d_u + y_v$, since the $A$--$v$ walk obtained by appending $v$ to the $A$--$u$ path that realizes $d_u$ satisfies $\sum_{\substack{F \in \f \\ F \cap Q \neq \emptyset}} y_F \leq d_u + y_v$. This implies that $d_v - y_v \leq d_u$, and by a symmetric argument, $d_u - y_u \leq d_v$. Consequently, we have $I_u \cap I_v \neq \emptyset$ for every edge $(u, v) \in E(G)$.

    Now, for every $r \in (0,1]$, we define $S_r = \{ v \in V(G) \mid r \in I_v \}$. 
    \begin{claim}\label{claim:S_r_is_A-B_sep}
        $S_r$ forms an $A$--$B$ separator for every $r\in (0,1]$.
    \end{claim}
    \begin{claimproof}
        Let $r\in (0,1]$ and $P = (v_1, v_2, \dots, v_p)$ be an $A$--$B$ path in $G$. Let $v_i$ be the first vertex on $P$ such that $d_{v_i} \geq r$. Such a vertex exists, since $v_p \in B$ and therefore $d_{v_p} \geq 1$, as $\{y_F\}_{F \in \f}$ is a solution to LP~\ref{lp:ab_separator_using_f}. If $i = 1$, then since $v_1 \in A$, we have $d_{v_1} = y_{v_1}$ which implies $v_i = v_1 \in S_r$. Otherwise, consider $v_{i-1}$. As the $A$--$v_i$ walk $Q$ obtained by appending $v_i$ to the $A$--$v_{i-1}$ path that realizes $d_{v_{i-1}}$ satisfies $\sum_{\substack{F \in \f \\ F \cap Q \neq \emptyset}} y_F \leq d_{v_{i-1}} + y_{v_i}$, we must have $d_{v_i} - y_{v_i} \leq d_{v_{i-1}} < r$. Consequently $r \in I_{v_i}$ which implies that $v_i \in S_r$, and hence $S_r$ is an $A$--$B$ separator.
    \end{claimproof}

    Consider the random process where an $r\in (0,1]$ is sampled uniformly at random.

    \begin{claim}\label{claim:S_r_has_small_cover}
        $\mathbb{E}_r[\fcov(S_r)]\ \leq\ 4k \cdot \log(2n) \cdot \sum_{F \in \f} y_F$.
    \end{claim}
    
    % \todo{(Julien) Maybe we should specify the random process?}
    
    \begin{claimproof}
        For each $F\in\f$ and $r\in(0,1]$ let us define $m^r_F = 0$ if $F \cap S_r = \emptyset$ and $\max\left( \{1\} \cup \{ \frac{1}{y_v} \, |\, v \in F \cap S_r \} \right)$ otherwise. Note that this quantity is well defined, since if $y_v = 0$, then $I_v = \emptyset$, and hence $r \notin I_v$, which implies $v \notin S_r$. Moreover, $m^r_F$ is at most $n$, because $y_F \geq \frac{1}{n}$ whenever $y_F \neq 0$, and consequently $y_v \geq \frac{1}{n}$ whenever $y_v \neq 0$. Also, let $\hat{y}_F^r = y_F \cdot m^r_F$. 
        We next show that $\{\hat{y}_F^r\}_{F \in \f}$ forms a fractional cover of $S_r$ for every $r \in (0, 1]$.
        Note that if $u \in S_r$, then
        $$
        \sum_{\substack{F \in \f\\ F \ni u}} \hat{y}_F^r \quad  
        \geq \quad \sum_{\substack{F \in \f\\ F \ni u}} y_F \cdot \max\left( \{1\} \cup \{ \frac{1}{y_v} \, |\, v \in F \cap S_r \} \right) \quad 
        \geq \quad y_u \cdot \max\, \{1, \frac{1}{y_u}\} \quad 
        \geq \quad 1.
        $$ 
        Therefore, 
        \begin{align}\label{eq:expectaion_A-B_separator}
            \nonumber \mathbb{E}_r[\fcov(S_r)] \quad
            &\leq \quad \mathbb{E}_r\left[\sum_{F \in \f} \hat{y}_F^r\right]\quad = \quad \sum_{F \in \f} y_F \cdot \mathbb{E}_r[m_F^r]\\
            \nonumber &\leq \quad \sum_{F \in \f} y_F \left[ \sum_{i = 0}^{\infty} 2^{i+1} \cdot \mathbb{P}_r[2^i \leq m_F^r < 2^{i+1}] \right]\\
            &\leq \quad \sum_{F \in \f} y_F \left[ \sum_{i = 0}^{\lfloor \log n \rfloor} 2^{i+1} \cdot \mathbb{P}_r[2^i \leq m_F^r] \right].
        \end{align} 
        
        Now, in order to bound $\mathbb{P}_r[2^i \leq m_F^r]$, for each $0 \leq i \leq \lfloor \log n \rfloor$ and $F \in \f$, we define $F_i := \{v\in F\ |\ y_v\leq 2^{-i}\}$. Let $u, v \in F_i$ and $P$ be the $u$--$v$ path in $G[F]$ of length at most $2k-1$ that uses only ancestors of $u$ or $v$, whose existence is asserted by Lemma~\ref{lem:ancestral_path}. Let $\Bar{v}$ be a vertex of this path, and without loss of generality, let $\Bar{v}$ be an ancestor of $v$. Since $v$ belongs to every set in $\f$ that contains $\Bar{v}$, we have $y_{\Bar{v}} \leq y_v$. Hence, $\Bar{v} \in F_i$ and consequently $P$ is a path in $G[F_i]$.
        % 
        % \todo{(Julien) $\Bar{v}$ is a slightly confusing variable name, especially since we assume it is an ancestor of $u$\\[5pt] (Ajay) Changed to $v$. Is that good?} 
        % 
        Now, we define $I_{F_i} := \bigcup_{v \in F_i} I_v$, let $u$ be the vertex in $F_i$ with the smallest value of $d_u - y_u$, and let $v$ be the vertex in $F_i$ with the largest value of $d_v$. Concatenating the $u$--$v$ path of length at most $2k-1$ in $G[F_i]$ with an $A$--$u$ path that realizes $d_u$ yields an $A$--$v$ walk. Since this $A$--$v$ walk contains an $A$--$v$ path, we conclude that $d_v \le d_u + (2k-1) \cdot 2^{-i}$.
        Thus, 
        $$
        \mathbb{P}_r[2^i \leq m_F^r] 
        \quad = \quad \mathbb{P}_r\left[\exists v \in F_i\ :\ r \in I_v \right] 
        \quad = \quad \mathbb{P}_r\left[r \in I_{F_i} \right] 
        \quad \leq \quad d_v - (d_u-2^{-i}) 
        \quad \leq \quad 2k \cdot 2^{-i}.
        $$
        
        Combining this with Equation \ref{eq:expectaion_A-B_separator}, we conclude that 
        $$
        \mathbb{E}_r[\fcov(S_r)] \quad
        \leq \quad \sum_{F \in \f} y_F \left[ \sum_{i = 0}^{\lfloor \log n \rfloor} 2^{i+1} \cdot 2k \cdot 2^{-i} \right] \quad
        \leq \quad 4k \cdot \log(2n) \cdot \sum_{F \in \f} y_F. 
        $$
    \end{claimproof}
    
    Now, as $\mathbb{E}_r[\fcov(S_r)] \leq 4k \cdot \log(2n) \cdot \sum_{F \in \f} y_F$ by Claim~\ref{claim:S_r_has_small_cover}, there must exist some $r_o\in(0,1]$ such that % the fractional cover of $S_{r_o}$ using sets from $\f$ must have value at most 
    $\fcov(S_{r_o}) \leq 4k \log(2n) \cdot \sum_{F \in \f} y_F$. Furthermore, $S_{r_o}$ is an $A$--$B$ separator by Claim~\ref{claim:S_r_is_A-B_sep}, which concludes our proof.
\end{proof}

Now we are ready to prove Theorem~\ref{thm:ab_separator}.

\ABSeparatorRounding*
\begin{proof}
    Since $\f$ is a layered family of $G$, we have $|\f| \leq n$. We invoke Lemma~\ref{lem:x_F_behaves_nicely} on the 5-tuple $(G, \f, A, B, \{x_F\}_{F \in \f})$ to obtain $\{y_F\}_{F \in \f}$, another solution to LP~\ref{lp:ab_separator_using_f} that satisfies $y_F \geq \frac{1}{|\f|} \geq \frac{1}{n}$ whenever $y_F > 0$, and $\sum_{F \in \f} y_F \leq 2 \cdot \sum_{F \in \f} x_F$. Since $(G,\f,A,B,\{y_F\}_{F \in \f})$ satisfies the requirements of Lemma~\ref{lem:ab_separator_with_conditions}, we are guaranteed the existence of an $A$--$B$ separator $S$ in $G$ such that $\fcov(S) \leq 4k \cdot \log(2n) \cdot \sum_{F \in \f} y_F \leq 8k \cdot \log(2n) \cdot \sum_{F \in \f} x_F$, which proves the theorem.
\end{proof}

\section{Using the Dual of A--B Separator LP}\label{sec:rounding_dual_ab_sep}
The main theorem that we prove in this section is the following:

\begin{theorem}\label{thm:ab_path_packing}
    Let $G$ be a graph, $k$ be a positive integer and $\f$ be a layered family of $G$ with thickness at most $k$. Let $A,B \subseteq V(G)$ be vertex subsets of $G$. 
    Then, for every $\ell \ge \frac{1}{6}\log(4|\f|)$, there exists a multiset $\q$ of induced $A$--$B$ paths in $G$ of cardinality at least $\ell f$, such that 
    \begin{itemize}
        \item For every $F \in \f$, the number of paths in $\q$ that intersect $F$ is at most $6(\ell+1)$.
        \item Each path in $\q$ intersects $F$ at most $2k-1$ times.
    \end{itemize}
    where $f$ denotes the optimum value of LP~\ref{lp:ab_separator_using_f} corresponding to $(G,\f,A,B)$.
\end{theorem}
\begin{proof}
    Since the theorem holds trivially if $f = 0$, let us assume that $f$ is positive. Note that due to the constraint in LP~\ref{lp:ab_separator_using_f}, $f$ has to be at least 1. Let $\p$ denote the set of all induced $A$--$B$ paths in $G$. We describe the dual of LP~\ref{lp:ab_separator_using_f} using non-negative real variables $\{y_P\}_{P\in\p}$. 
    \begin{align}\label{lp:ab_path_packing}
        \text{Maximize :}&\quad \sum\limits_{P\in\p} y_{P}\\
        \nonumber \text{Subject To :}&\quad \sum\limits_{\substack{P\in \p\\ P\cap F\neq \emptyset}} y_{P} \leq 1 && \forall F\in \f
    \end{align}
    Observe that, by strong duality~\cite{doi:10.1137/1025101}, the dual LP also has the same optimum value $f$ as the primal. 
    
    \begin{claim}
        There exists an optimal solution to LP~\ref{lp:ab_path_packing} where, for every $P\in\p$ we have $y_P > 0$ only if $P$ is upward minimal.
    \end{claim}
    \begin{claimproof}
        Consider the optimal solution $\{y_P\}_{P\in\p}$ to LP~\ref{lp:ab_path_packing} with the smallest number of variables that have a non-zero assignment to upward non-minimal paths. For the sake of contradiction, let $P'$ be an upward non-minimal path with $y_{P'} > 0$ and let $P_o$ be any upward minimal path that witnesses its upward non-minimality, i.e., let $P_o$ be upward minimal and satisfy $\widehat{P}_o \subsetneq \widehat{P}'$, or $\widehat{P}_o = \widehat{P}'$ and $|V(P_o)| < |V(P')|$. %\todo{(julien) do we need that $P_o$ to be upward minimal? (as $y_{P_o}$ might be $0$) }\todo{(Ajay) Yes (the contradiction needs it), should we justify why such a $P_o$ exists or is it okay to leave it as is?}\todo{(julien): I think it's fine}
        We claim that the assignment $\{y'_P\}_{P \in \p}$, defined by
        $$
        y'_P = 
        \begin{cases}
        y_{P_o} + y_{P'}, & \text{if } P = P_o, \\
        0, & \text{if } P = P',\\
        y_P, & \text{otherwise}, \\
        \end{cases}
        $$
        is a feasible solution. This contradicts the definition of $\{y_P\}_{P \in \p}$, since $\{y'_P\}_{P \in \p}$ attains the same (optimal) objective function value but has one fewer nonzero variable that corresponds to an upward non-minimal path.

        To this end, observe that any constraint of the LP that does not involve $y'_{P_o}$ remains satisfied, since the values of all other variables in $\{y'_P\}_{P \in \p}$ do not exceed those in $\{y_P\}_{P \in \p}$. Now consider a constraint that involves $y'_{P_o}$ and let $F$ be the element of $\f$ it corresponds to. If $P_o \cap F \neq \emptyset$, then $\widehat{P}_o \cap F \neq \emptyset$, and consequently $\widehat{P}' \cap F \neq \emptyset$ as well. Moreover, if $\widehat{P}' \cap F \neq \emptyset$, then as elements of $\f$ are downward closed, $F$ must contain a vertex of $P'$, and hence $F \cap P' \neq \emptyset$. Therefore, any constraint that contains $y'_{P_o}$ also contains $y'_{P'}$, and since $y'_{P_o} + y'_{P'} = y_{P'} + y_{P_o}$, these constraints also remain satisfied. Thus, all constraints remain satisfied under $\{y'_P\}_{P \in \p}$, which concludes the proof of the claim.
    \end{claimproof}
    
    Let $\{y_P\}_{P\in\p}$ be an optimal solution to LP~\ref{lp:ab_path_packing}, satisfying $y_P > 0$ only if $P$ is upward minimal. Observe that $\mathcal{D} := \{\,y_P/f\,\}_{P\in\p}$ is a probability distribution over $\p$, as every $y_P$ takes non-negative values and satisfy $\sum_{P\in\p}y_P = f$. Let $\q := \{Q_i\}_{i\in[\,f \cdot \ell\,]}$ be $\lceil f \cdot \ell\rceil$ independent samples drawn from the distribution $\mathcal{D}$ over the set $\p$. Let $F \in \f$, and let $\chi_F$ denote the number of paths in $\q$ that intersect $F$. For every $i\in[\,f \cdot \ell\,]$, we have that,
    $$
    \displaystyle \mathop{\mathbb{P}}\limits_{\substack{Q_i\sim\mathcal{D}}}\left[\,Q_i \cap F \neq \emptyset\,\right] \quad
    = \quad \sum_{\substack{P\in\p\\ P\cap F\neq \emptyset}}
    \displaystyle \mathop{\mathbb{P}}\limits_{\substack{Q_i\sim\mathcal{D}}}\left[\, Q_i = P\, \right] \quad
    = \quad \sum_{\substack{P\in\p\\ P\cap F\neq \emptyset}} \frac{y_{P}}{f} \quad
    \leq \quad \frac{1}{f}.
    $$
    By linearity of expectation, combined with the fact that $f \geq 1$, it follows that the expected value of $\chi_F$ is at most $\ell + 1$. Applying union bound over all $F\in\f$ and the Chernoff bound from Theorem \ref{thm:chernoff} to $\chi_F$ for every $F\in\f$, we get:
    $$
    \mathbb{P}\left[\,\max_{F\in\f}\{\chi_F\}\; <\; 6 (\ell+1)\,\right] \quad
    \geq \quad 1\ -\ \sum_{F\in\f}\mathbb{P}\left[\,\chi_F\; \geq\; 6 (\ell+1)\,\right] \quad
    \geq \quad 1\ -\ |\f| \cdot 2^{-6 (\ell+1)} \quad
    \geq \quad \frac{3}{4}.
    $$
    Hence, with probability at least $\frac{3}{4}$, the sampled multiset has cardinality at least $f \cdot \ell$, consists of induced $A$--$B$ paths in $G$, and satisfies that for every $F \in \f$, at most $6(\ell+1)$ of these paths intersect $F$. Therefore, there must exist some multiset $\q$ with these properties. Finally, observe that $\q$ can only contain upward minimal paths since $\mathcal{D}$ assigns non-zero probability only to them. This, combined with Lemma~\ref{lem:minimal_paths_are_short}, implies the statement of the theorem.
\end{proof}

\section{Rounding Balanced Separator LP}\label{sec:rounding_balanced_sep}

We begin this section by defining the \emph{balanced separator LP}. Let $G$ be a graph, $\f$ be a layered family of $G$, $\cf$ be the set of centers of elements in $\f$ and $X \subseteq \cf$. For each pair $u, v \in X$, let $\p_{u,v}$ denote the set of all induced paths from $u$ to $v$ in $G$. We describe the balanced separator LP corresponding to the instance $(G, X, \f)$, using non-negative real variables $x_F$ and $d_{u,v}$, defined for every $F \in \f$ and $u, v \in X$.
\begin{align}\label{lp:container_cover}
    \text{Minimize:} \quad & \sum_{F \in \f} x_F \\
    \nonumber
    \text{Subject to:} \quad 
    & \sum_{v \in X} d_{u,v} \geq \frac{|X|}{10} && \forall u \in X \\
    \nonumber
    & d_{u,v} \leq \sum_{\substack{F \in \f \\ F \cap P \neq \emptyset}} x_F && \forall u,v \in X,\ P \in \p_{u,v} \\
    \nonumber
    & d_{u,v} \leq 1 && \forall u,v \in X
\end{align}
\newline
The linear program above is feasible and bounded: setting $x_F = 1$ for every $F \in \mathcal{F}$ and $d_{u,v} = 1$ for every $u, v \in X$ is a feasible assignment, since $V(G) \subseteq \bigcup_{F \in \mathcal{F}} F$ ensures that every induced path from $u$ to $v$ intersects some $F \in \mathcal{F}$. We also remark that it is a variant of the one used in the Leighton–Rao rounding scheme~\cite{leighton1999multicommodity}, adapted to operate on vertex sets rather than individual vertices, as in~\cite{KorchemnaL0S024,chudnovsky2025treewidth}. We are now ready to state the main theorem proved in this section:

\begin{restatable}{theorem}{LPoptVSintegral}\label{thm:balanced_separator_rounding}
Let $G$ be a graph, $k$ be a positive integer and $\f$ be a layered family of $G$ with thickness at most $k$. Then for every subset $X \subseteq \cf$, where $\cf$ denotes the set of centers of the elements of $\f$, there exists an $(X, \tfrac{95}{100})$-balanced separator $S$ in $G$ such that $S$ is downward closed and
$$
\mathrm{fcov}_{\f}(S)\ \leq\ 5000kf  \log(2n) \log (f+4),
$$
whenever $f$, which denotes the optimum value of the balanced separator LP corresponding to $(G, X, \mathcal{F})$, is positive.
\end{restatable}

For the remainder of this section, we fix a graph $G$, the layered family $\f$ and the set $X\subseteq \cf$. We also fix an optimal solution to LP~\ref{lp:container_cover} corresponding to $(G, X, \f)$ of value $f>0$. Let $x_F$ and $d_{u,v}$, for every $F \in \f$ and $u,v\in X$, denote the values assigned to the corresponding variables in this solution.

For $U \subseteq V(G)$, we define $\mu(U) := \sum_{\substack{F \in \f \\ F \cap U \neq \emptyset}} x_F$. 
For every $v \in V(G)$, we define $x_v := \sum_{\substack{F \in \f \\ F \ni v}} x_F$. Also, for every $u, v \in V(G)$, we define $d(u,v) := \min\,\{\, \sum_{\substack{F \in \f \\ F \cap P \neq \emptyset}} x_F \,|\, P \in \p_{u,v}\, \}$.%\todo{does $d(v,v)=x_v$? Seems like it, as it is used implicitly at the end of Lemma 6.7 and in 6.6.1. In any case, the edge case should be clearer.}\todo{(Ajay) Yes, Done} 
Note that $d(\cdot,\cdot)$ is a symmetric function, and that $d(v,v)=x_v$. We now make some more observations about the function $d$.

% \todo[inline]{(Julien): often $\bar{u}$ is used instead of a simple $u$ and I don't know why. I think it would make the notation better to drop the bar.}\todo[inline]{(Ajay): $\bar{u}$ is used (consistently, will check once more) to represent the center of the ball. Did this so I have $u$ available to use for arbitrary vertices.}
\begin{observation}\label{obs:distance_defintions_match}
    If $u, v \in X$, then $\min\{d(\bar{u}, v),1\} \geq d_{u,v}$.
\end{observation}

\begin{observation}\label{obs:triangle_inequality}
    If $u,v,w \in V(G)$, then $d(u,w) \leq d(u,v) + d(v,w)$. In particular, if $(v,w) \in E(G)$ and $u \in V(G)$, then $d(u,w) \leq d(u,v) + x_w$. 
\end{observation}

Observation~\ref{obs:distance_defintions_match} follows immediately from the LP constraints: \emph{$d_{u,v} \leq \sum_{\substack{F\in \f\\ F\cap P \neq\emptyset}} x_F$ and $d_{u,v} \leq 1$ for every $u,v\in X$ and $P\in\p_{u,v}$}. Furthermore, note that $d(u,v) \leq \sum_{\substack{F \in \f \\ F \cap Q \neq \emptyset}} x_F$ for any $u$--$v$ walk $Q$ in $G$, as we can always find an induced $u$--$v$ path in $G[Q]$. Hence, Observation~\ref{obs:triangle_inequality} follows from the fact that if $P_1$ and $P_2$ are induced $u$--$v$ and $v$--$w$ paths realizing $d(u,v)$ and $d(v,w)$ respectively, then appending $P_2$ to $P_1$ yields a $u$--$w$ walk $Q$ satisfying $\sum_{\substack{F \in \f \\ F \cap Q \neq \emptyset}} x_F \leq d(u,v) + d(v,w)$. Similarly, if $(v,w) \in E(G)$, then appending $w$ to the $u$--$v$ path realizing $d(u,v)$ produces a $u$--$w$ walk satisfying $\sum_{\substack{F \in \mathcal{F} \\ F \cap Q \neq \emptyset}} x_F \leq d(u,v) + x_w$.\\

Let $\varepsilon := \frac{1}{500 \log (f + 4)}$, and for every $u \in V(G)$, $C\subseteq V(G)$ and positive real number $r$, define 
$$
B_C(u, r) := \{v \in C \mid d(u, v) \leq r\}
\quad \text{and} \quad
\delta_C(u, r) := B_C(u, r + 3\varepsilon) \setminus B_C(u, r + \varepsilon).
$$
Let $\ell_{\max} = \lceil \log \frac{ f + 4}{\varepsilon}\rceil$ and define $r_{i} := 3i \varepsilon$, where $i$ is a positive integer. Let $Z_0 := \{\, v \in V(G) \mid x_v \ge \tfrac{\varepsilon}{2k} \,\}$.
We note that every $F \in \f$ that contains a vertex $u\in Z_0$ also contains every descendant of $u$. This implies that $x_v \geq x_u$ and consequently $v\in Z_0$ for every descendant $v$ of $u$. Hence $Z_0$ is downward closed.

% \todo[inline]{say we use the following lemma many times}

\begin{lemma}\label{lem:F_is_local}
    Let $u$ and $v$ be vertices in $V(G)\setminus Z_0$ that satisfy $d(u,v) \geq \varepsilon$. Then there does not exist $F\in \f$ that contains both $u$ and $v$.
\end{lemma}
\begin{proof}
    For the sake of contradiction, suppose there exists $F\in\f$ that contains both $u$ and $v$. Let $P$ be the $u$--$v$ path in $G[F]$ of length at most $2k-1$ that uses only ancestors of $u$ or $v$, whose existence is asserted by Lemma~\ref{lem:ancestral_path}. Let $\Bar{v}$ be a vertex of this path and without loss of generality let $\Bar{v}$ be an ancestor of $v$. Since $v$ belongs to every set in $\f$ that contains $\Bar{v}$, we have $x_{\Bar{v}} \leq x_v \leq \max\,\{x_u,x_v\}$. Hence $\Bar{v} \in V(G) \setminus Z_0$, and consequently $P$ is a path in $G-Z_0$. Thus, we have
    % \todo{(Julien) $x_u$?\\[5pt] (Ajay) Edited}
    $$
    d(u, v) \ \
    \leq \ \ \sum_{\substack{F \in \f \\ F \cap P \neq \emptyset}} x_F \ \
    < \ \ (2k-1)\cdot \frac{\varepsilon}{2k} \ \
    < \ \ \varepsilon,
    $$
    which gives the desired contradiction. 
\end{proof}

Now we state and prove three properties in Lemma~\ref{lem:good_layer_many_outside},~\ref{lem:good_layer_cost_and_gain} and~\ref{lem:good_layer_frac_separator} that we require to implement a version of the Leighton-Rao rounding scheme~\cite{leighton1999multicommodity} and prove Theorem~\ref{thm:balanced_separator_rounding}.

% \todo{We should probably add a reference\\[5pt] (Ajay)\\ Done}

\begin{lemma}\label{lem:good_layer_many_outside}
    Let $Z \subseteq V(G)$ be a superset of $Z_0$. Suppose that $G - Z$ contains a connected component $C$ such that $|C \cap X| > \frac{95|X|}{100}$, and let $\Bar{u} \in X \cap C$ and $\ell \in [\,\ell_{\max}\,]$. Then we have,
    $$
    \left| X \setminus B_C(\Bar{u}, r_{\ell+1}) \right| \geq \frac{5|X|}{100}.
    $$
\end{lemma}
% \todo[inline]{(Julien) I don't think the assumption about $Z$ is used}
% \todo[inline]{(Ajay) Actually I don't think we use the property of $C$ either. I kept all three lemmas uniform because then Lemmas~\ref{lem:good_layer_many_outside},~\ref{lem:good_layer_cost_and_gain} and~\ref{lem:good_layer_frac_separator} all look the same. So someone who's reading don't have to context-switch while moving from one lemma to the other and can save RAM.}
\begin{proof}
    Since $\ell \in [\,\ell_{\max}\,]$, we have
    $$
    r_{\ell+1}
    \ \ \le \ \ r_{\ell_{\max}+1}
    \ \ = \ \ 3\varepsilon\bigl(\lceil \log \tfrac{f+4}{\varepsilon}\rceil + 1\bigr)
    \ \ \le \ \ \frac{3\log \bigl(2000 \cdot (f+4) \cdot \log (f+4)\bigr)}{500 \log (f+4)}
    \ \ \le \ \ \frac{1}{20},
    $$
    where the third inequality upper bounds $\lceil \log \tfrac{f+4}{\varepsilon}\rceil + 1$ by $\log \tfrac{4(f+4)}{\varepsilon}$ and substitutes the value of $\varepsilon$ and the final inequality follows from the facts that for $x \ge 4$ we have $\log x \ge 2$ and $\frac{\log\log x}{\log x} \le \tfrac{3}{4}$. Now, assume, for the sake of contradiction, that $\left| X \setminus B_C(\Bar{u}, r_{\ell+1}) \right| < \frac{5|X|}{100}$.
    Then we obtain:
    $$
    \begin{aligned}
    \sum_{v \in X} \min\{d(\bar{u}, v),1\} \quad
    &=\quad \sum_{v \in X \cap B_C(\bar{u}, r_{\ell+1})} \min\{d(\bar{u}, v),1\} 
      \quad + \sum_{v \in X \setminus B_C(\bar{u}, r_{\ell+1})} \min\{d(\bar{u}, v),1\} \\
    &\leq \quad\left(|X| - |X\setminus B_C(\Bar{u}, r_{\ell+1})|\right)\cdot r_{\ell+1} + |X\setminus B_C(\Bar{u}, r_{\ell+1})|\cdot 1 \\
    &<\quad \frac{95|X| \cdot r_{\ell+1}}{100}\ +\ \frac{5|X| \cdot 1}{100} \\
    &<\quad \frac{|X|}{10}.
    \end{aligned}
    $$
    % \todo{i think this sum should be min of $d(u,v)$ and $1$}
    Here the second inequality follows from the fact that $r_{\ell+1} < 1$ and consequently maximizing $\left(|X| - |X\setminus B_C(\Bar{u}, r_{\ell+1})|\right)\cdot r_{\ell+1} + |X\setminus B_C(\Bar{u}, r_{\ell+1})|\cdot 1$ is equivalent to maximizing $|X\setminus B_C(\Bar{u}, r_{\ell+1})|$.
    But, Observation~\ref{obs:distance_defintions_match} implies the inequality $\sum_{v \in X} \min\{d(\bar{u}, v),1\} \geq \sum_{v \in X} d_{\Bar{u},v}$ which contradicts the constraint $\sum_{v \in X} d_{\Bar{u}, v} \geq \frac{|X|}{10}$ of LP~\ref{lp:container_cover}. Thus, it follows that $|X \setminus B_C(\Bar{u}, r_{\ell+1})| \geq \frac{5|X|}{100}$.
\end{proof}

% \todo[inline]{say: we will show in the next lemma that the amount we pay is dependent on the boundary, and that this lemma says that there exists a good layer where the cost of the boundary is related to the cost of the thing we peel away}

\begin{lemma}\label{lem:good_layer_cost_and_gain}
    Let $Z \subseteq V(G)$ be a superset of $Z_0$. Suppose that $G - Z$ contains a connected component $C$ such that $|C \cap X| > \frac{95|X|}{100}$, and let $\Bar{u} \in X \cap C$. Then there exists $\ell \in [\,\ell_{\max}\,]$ such that $\mu(\delta_C(\Bar{u}, r_\ell)) \leq \mu(B_C(\Bar{u}, r_\ell))$.
\end{lemma}
\begin{proof}
    For this, we need the following claims:
    
    \begin{claim}\label{claim:inner_ball_has_non_zero_measure}
        $\mu(B_C(\Bar{u}, r_1)) \geq 2\varepsilon$.
    \end{claim}
    \begin{claimproof}
        Consider a path $P$ in $G[C]$ from $\Bar{u}$ to a vertex $w$ such that $w \in X\cap (C\setminus B_C(\Bar{u}, r_{\ell_{\max} + 1}))$. Note that since $| X \setminus B_C(\Bar{u}, r_{\ell+1}) | \geq \frac{5|X|}{100}$ by Lemma~\ref{lem:good_layer_many_outside} and since $|X \setminus C| < \frac{5|X|}{100}$ by assumption, such a vertex exists. %\todo{Why does it follow from 6.5?}\todo{added explanation}
        Let $P' := (\Bar{u},v_1,\dots,v_p)$ be the subpath of $P$ such that the successor of $v_p$ in $P$ (say $v_{p+1}$) is the first vertex in $P$ that does not belong to $B_C(\Bar{u}, r_1)$. $P'$ is well defined, since $w$ does not belong to $B_C(\Bar{u}, r_1)$ and is non-empty since $\Bar{u} \in B_C(\Bar{u}, r_1)$. %\todo{why?}\todo{(Ajay): @Julien Is the confusion about why $\Bar{u} \in B_C(\Bar{u}, r_1)$ is true? Maybe I can add $d(\bar{u},\bar{u}) = x_{\bar{u}} \leq \varepsilon < r_1$?}
        Hence, by Observation~\ref{obs:triangle_inequality}, we get:
        $$
        \mu(B_C(\Bar{u}, r_1)) 
        % \ \ \geq \ \ \sum_{\substack{F \in \f \\ F \cap B_C(\Bar{u}, r_1) \neq \emptyset}} x_F 
        \ \ \geq \ \ \sum_{\substack{F \in \f \\ F \cap P' \neq \emptyset}} x_F 
        \ \ \geq \ \ d(\Bar{u}, v_p) 
        \ \ \geq \ \ d(\Bar{u}, v_{p+1}) - x_{v_{p+1}}
        \ \ \geq \ \ 3\varepsilon - \frac{\varepsilon}{2k}
        \ \ > \ \ 2\varepsilon.
        $$ 
    \end{claimproof}
    
    \begin{claim}\label{claim:meause_of_ball_and_boundary_is_additive}
        $\mu(B_C(\Bar{u}, r_{\ell+1})) \geq \mu(B_C(\Bar{u}, r_\ell)) + \mu(\delta_C(\Bar{u}, r_\ell))$ for every $\ell \in [\,\ell_{\max}\,]$.
    \end{claim}
    \begin{claimproof}
        Let $v \in \delta_C(\bar{u}, r_\ell)$, and let $F \in \f$ be a set containing $v$, with $w$ denoting the center of $F$. Observe that $d(\bar{u},v) > r_\ell + \varepsilon$ while $d(\bar{u},v') \leq r_\ell$ for every $v'\in B_C(\Bar{u}, r_\ell)$. Hence by Observation~\ref{obs:triangle_inequality}, we get $d(v',v) \geq \varepsilon$, and by Lemma~\ref{lem:F_is_local} we conclude that there does not exist an $F\in\f$ that contains both $v$ and $v'$. As every set $F \in \f$ that intersects $\delta_C(\bar{u}, r_\ell)$ has an empty intersection with $B_C(\bar{u}, r_\ell)$, it follows that no set in $\f$ contributes to both $\mu(\delta_C(\bar{u}, r_\ell))$ and $\mu(B_C(\bar{u}, r_\ell))$. Moreover, from the definitions it follows that every set $F \in \f$ contributing to $\mu(B_C(\bar{u}, r_\ell)) + \mu(\delta_C(\bar{u}, r_\ell))$ also contributes to $\mu(B_C(\bar{u}, r_{\ell+1}))$. These two observations together imply the claim.
    \end{claimproof}

    Now, combining Claims~\ref{claim:inner_ball_has_non_zero_measure} and~\ref{claim:meause_of_ball_and_boundary_is_additive}, we conclude that if $\mu(\delta_C(\Bar{u}, r_\ell)) > \mu(B_C(\Bar{u}, r_\ell))$ for every $\ell \in [\,\ell_{\max}\,]$, then $\mu(B_C(\Bar{u}, r_{\ell_{\max}})) > 2^{\ell_{\max} - 1} \cdot \mu(B_C(\Bar{u}, r_1)) \geq f$, while $\mu(G) \leq f$, which yields a contradiction. Hence we conclude that there exists some $\ell \in [\,\ell_{\max}\,]$ for which we have $\mu(\delta_C(\Bar{u}, r_\ell)) \leq \mu(B_C(\Bar{u}, r_\ell))$, which proves Lemma~\ref{lem:good_layer_cost_and_gain}.
\end{proof}

% \todo[inline]{say: this lemma relates the amount we pay to the size of the boundary, but in order to apply ab separator theorem, we need to state size in terms of the induced layered family}
 
\begin{lemma}\label{lem:good_layer_frac_separator}
    Let $Z \subseteq V(G)$ be a downward closed superset of $Z_0$. Suppose that $G - Z$ contains a connected component $C$ such that $|C \cap X| > \frac{95|X|}{100}$, and let $\Bar{u} \in X \cap C$ and $\ell \in [\,\ell_{\max}\,]$. If $A := B_C(\Bar{u}, r_\ell + \varepsilon)$, $B := C \setminus B_C(\Bar{u}, r_{\ell+1})$ and $\f' := \{\,F \cap C \mid F \in \f \text{ and } F \cap C \neq \emptyset\,\}$, then there exists a fractional $(A,B)$-separator using $\f'$ in $G[C]$ of weight at most $\frac{1}{\varepsilon}\cdot\mu(\delta_C(\Bar{u}, r_\ell))$.
\end{lemma}

\begin{proof}
    Consider the assignment $\{x'_{F'}\}_{F' \in \f'}$ defined as follows. Let $F \in \f$ and let $F'\in\f'$ be the set corresponding to $F\cap C$. If $F \cap \delta_C(\bar{u}, r_\ell)$ is nonempty, define $x'_{F'} := \frac{x_F}{\varepsilon}$; otherwise, define $x'_{F'} := 0$. We claim that $\{x'_{F'}\}_{F' \in \f'}$ is a fractional $(A, B)$-separator using $\f'$ in $G[C]$.

    Let $P= (v_0,v_1, \dots, v_p,v_{p+1})$ be a minimal (under subpaths) $A$--$B$ path in $G[C]$. We have that $P\cap A = \{v_0\}$ and $P\cap B = \{v_{p+1}\}$. Let $P^*=(v_1, \dots, v_p)$ be the interior of $P$. By definition, we have $V(P^*) \subseteq \delta_C(\bar{u}, r_\ell)$. Moreover, $P^*$ is non-empty; otherwise, applying Observation~\ref{obs:triangle_inequality} to $\{\Bar{u}, v_1, v_{p+1}\}$ yields $d(\Bar{u}, v_{p+1}) \leq d(\Bar{u}, v_{1}) + x_{v_{p+1}} \leq r_\ell + \varepsilon + \frac{\varepsilon}{2k}$, which contradicts the assumption that $d(\Bar{u}, v_{p+1}) > r_\ell + 3\varepsilon$. 
    $$
    \begin{aligned}
        \sum_{\substack{F' \in \f'\\ F'\cap P\neq \emptyset}} x'_{F'} \quad 
        & \ge \quad \sum_{\substack{F \in \f \\ F \cap P^* \neq \emptyset}} \frac{x_{F}}{\varepsilon}
        \quad \ge \quad \frac{1}{\varepsilon}\cdot d(v_1,v_{p}) \\[2pt]
        & \ge\ \quad \frac{1}{\varepsilon}\cdot \bigl(d(\Bar{u},v_{p})-d(\Bar{u},v_1)\bigr) \\
        & \ge\ \quad \frac{1}{\varepsilon} \left((r_\ell + 3\varepsilon - \frac{\varepsilon}{2k}) - (r_\ell + \varepsilon + \frac{\varepsilon}{2k})\right)
        \quad \ge \quad 1,
    \end{aligned}
    $$
    where the first inequality follows from the fact that any set in $\f$ that intersects $P^*$ must intersect $\delta_C(\bar{u}, r_\ell)$. The second line follows from applying Observation~\ref{obs:triangle_inequality} to the triple $\{\Bar{u}, v_1, v_{p}\}$. The last applies Observation~\ref{obs:triangle_inequality} to the triples $\{\Bar{u}, v_{p}, v_{p+1}\}$ and $\{\Bar{u},v_0,v_1\}$ and uses $v_{1}, v_{p}\not\in Z_0$. The claimed upper bound on the weight of the fractional separator follows from the fact that $\sum_{F \in \f'} x'_{F}  = \frac{1}{\varepsilon}\cdot\mu(\delta_C(\Bar{u}, r_\ell))$.
\end{proof}

With this, we prove the key lemma used in the proof of Theorem~\ref{thm:balanced_separator_rounding}.

% \todo[inline]{say: this is one iteration of the peeling away process.}

\begin{lemma}\label{lem:cutoff_procedure}
     Let $Z\subseteq V(G)$ be a downward closed superset of $Z_0$. If $G-Z$ has a connected component $C$ such that $|C\cap X| > \frac{95|X|}{100}$ then there exists a partition $A\cup S\cup B$ of $C$ such that:
    \begin{enumerate}
        \item $S$ is downward closed.
        \item $S$ is an $A$--$B$ separator in $G-Z$.
        \item $B \subsetneq C$.
        \item $|A \cap X| \leq \frac{95|X|}{100}$.
        \item $\mathrm{fcov}_{\f}(S) \leq 4000k \cdot \log(2n) \cdot \log (f + 4) \cdot (\mu(C) - \mu(B))$.
    \end{enumerate}
\end{lemma}
\begin{proof}
    Observe that by Lemma~\ref{lem:layered_family_minus_downward_closed}, $ \f'' := \{F\setminus Z\ |\ F\in\f \text{ and } F\setminus Z\neq\emptyset\}$ is a layered family of $G-Z$ with thickness at most $k$. Furthermore, by Lemma~\ref{lem:layered_family_restricted_to_component} applied to the component $C$ in $G-Z$ and layered family $\f''$, we conclude that 
    $$\f'\ \ :=\ \ \{F'' \cap C\ |\ F''\in\f'' \text{ and } F''\cap C\neq\emptyset\}\ \ =\ \ \{F\cap C\ |\ F\in\f \text{ and } F\cap C\neq\emptyset\},$$ 
    is a layered family of $G[C]$ with thickness at most $k$.
    Let $\Bar{u}\in C\cap X$. We apply Lemma~\ref{lem:good_layer_cost_and_gain} to obtain some $\ell \in [\,\ell_{\max}\,]$ such that $\mu(\delta_C(\Bar{u}, r_\ell)) \leq \mu(B_C(\Bar{u}, r_\ell))$. Now, let $A' := B_C(\Bar{u}, r_\ell + \varepsilon)$ and $B' := C \setminus B_C(\Bar{u}, r_{\ell+1})$ .
    We also know, by Lemma~\ref{lem:good_layer_frac_separator}, that there exists a fractional $(A',B')$-separator using $\f'$ in $G[C]$ of weight at most $\frac{1}{\varepsilon}\cdot\mu(\delta_C(\Bar{u}, r_\ell))$. Therefore, we can apply Theorem~\ref{thm:ab_separator} to the tuple $(G[C], A', B', \f')$ to obtain an $A'$--$B'$ separator $S'$ in $G[C]$.
    
    Define $S$ to be equal to $\widecheck{S'}$, $B$ to be the set of vertices in connected components of $G[C \setminus S]$ that have a non-empty intersection with $B'$, and let $A := C \setminus (B \cup S)$. Observe that, by definition, $S$ is both an $A$--$B'$ separator and an $A'$--$B$ separator in $G[C]$, and consequently also in $G - Z$. This establishes two facts. First, we have $|A \cap X| \le |X| - |B' \cap X|$. By Lemma~\ref{lem:good_layer_many_outside}, $|B' \cap X| \ge \tfrac{5|X|}{100}$, which implies $|A \cap X| \le \tfrac{95|X|}{100}$. Second, we obtain that $B \subsetneq C$ as $\bar{u}\in A'$ and consequently $\bar{u}\in C\setminus B$.
    %\todo{why?}\todo{Added: $\bar{u}\in A'$ and $S$ is an $A'$--$B$ separator.}
    % \todo{add this notation to prelims} 
    Now, we bound the fractional cover number of $S$ using $\f$ as follows:
    $$
    \mathrm{fcov}_{\f}(S)
    \ \ \leq \ \ \mathrm{fcov}_{\f'}(S')
    \ \ \leq \ \ 8k \cdot \log(2n) \cdot \frac{1}{\varepsilon} \mu(\delta_C(\Bar{u}, r_\ell)) 
    \ \ \leq \ \ 4000k \cdot \log(2n) \cdot \log (f + 4) \cdot \mu(B_C(\Bar{u}, r_\ell)).
    $$
    Here, the first inequality follows from two observations. First, any fractional cover of $S'$ using $\f'$ also covers $S$: every set containing a vertex of $S'$ also contains all of its descendants, which implies that if a vertex of $S'$ is covered, then all of its descendants are as well. Second, any set covered by $\f'$ can also be covered by $\f$, since $\f'$ is the restriction of $\f$ to a subset of $V(G)$. The second inequality follows from a property of $S$ guaranteed by Theorem~\ref{thm:ab_separator}, while the last follows from our choice of $\ell$ and definition of $\varepsilon$.
    
    To conclude the proof, it remains to show that $\mu(B_C(\bar{u}, r_\ell)) \le \mu(C) - \mu(B)$. To this end, let $F \in \f$ be a set intersecting $B$, and let $v \in F \cap B$. Since $v \in B$ implies $v \notin A'$, we have $d(\bar{u}, v) > r_\ell + \varepsilon$, while $d(\bar{u}, v') \le r_\ell$ for every $v' \in B_C(\bar{u}, r_\ell)$. Hence, by Observation~\ref{obs:triangle_inequality}, it follows that $d(v', v) \ge \varepsilon$. By Lemma~\ref{lem:F_is_local}, no set $F \in \f$ can contain both $v$ and $v'$. Therefore, no set in $\f$ contributes to both $\mu(B)$ and $\mu(B_C(\bar{u}, r_\ell))$. Consequently $\mu(B_C(\bar{u}, r_\ell)) \le \mu(C) - \mu(B)$, which completes the proof.
\end{proof}

Now we are ready to prove Theorem~\ref{thm:balanced_separator_rounding}.

\LPoptVSintegral*
\begin{proof}
    We prove the following claim:\\[-4pt]
    
    \noindent\textbf{Claim~6.1.1.}\label{claim:balanced_sep_cutoff_step}
        \emph{Let $Z\subseteq V(G)$ be a downward closed superset of $Z_0$, and let $C$ be the connected component of $G - Z$ such that $|C \cap X| > \frac{95|X|}{100}$, or let $C = \emptyset$ if no such component exists. Then there exists an $(X, \frac{95}{100})$-balanced separator $S$ in $G$ such that $S$ is downward closed and $\mathrm{fcov}_{\f}(S) \leq \mathrm{fcov}_{\f}(Z) + 4000k \cdot \log(2n) \cdot \log (f + 4) \cdot \mu(C)$.}
    % We prove the following claim:
    % \begin{claim}\label{claim:balanced_sep_cutoff_step}
    %     Let $Z\subseteq V(G)$ be a downward closed superset of $Z_0$, and let $C$ be the connected component of $G - Z$ such that $|C \cap X| > \frac{95|X|}{100}$, or let $C = \emptyset$ if no such component exists. Then there exists an $(X, \frac{95}{100})$-balanced separator $S$ in $G$ such that $S$ is downward closed and $\mathrm{fcov}_{\f}(S) \leq \mathrm{fcov}_{\f}(Z) + 4000k \cdot \log(2n) \cdot \log (f + 4) \cdot \mu(C)$.
    % \end{claim}
    \begin{claimproof}
        We proceed by induction on $|C|$. Firstly, if $C = \emptyset$, then we can let $S = Z$ and the lemma holds true. Otherwise, let $A'\cup S'\cup B'$ be the partition of $C$ whose existence is implied by Lemma~\ref{lem:cutoff_procedure}. Observe that $Z\cup S'$ is downward closed, as both $Z$ and $S'$ are. Now let us consider the graph $G-(Z\cup S')$ and distinguish two cases. First, if there exists a component $C'$ of $G - (Z \cup S')$ such that $|C' \cap X| > \frac{95|X|}{100}$, then $C' \subseteq B'$ as $S'$ is an $A'$--$B'$ separator in $G - Z$ and $|A' \cap X| \leq \frac{95|X|}{100}$. Since $B' \subsetneq C$, we conclude that $C' \subsetneq C$. Second, no such component exists; hence we set $C'=\emptyset$, and again $C' \subsetneq C$. In either case, by applying the inductive hypothesis to the pair $(G, Z \cup S')$, we obtain an $(X, \frac{95}{100})$-balanced separator $S$ in $G$ such that $S$ is downward closed and
        \begin{align*}
            \mathrm{fcov}_{\f}(S) \quad
            &\leq \quad \mathrm{fcov}_{\f}(Z \cup S')\ +\ 4000k \cdot \log(2n) \cdot \log (f + 4) \cdot \mu(C') \\
            &\leq \quad \mathrm{fcov}_{\f}(Z)\ +\ \mathrm{fcov}_{\f}(S')\ +\ 4000k \cdot \log(2n) \cdot \log (f + 4) \cdot \mu(C') \\
            &\leq \quad \mathrm{fcov}_{\f}(Z)\ +\ 4000k \cdot \log(2n) \cdot \log (f + 4) \cdot \left(\mu(C)\ - \mu(B') +\ \mu(C')\right) \\
            &\leq \quad \mathrm{fcov}_{\f}(Z)\ +\ 4000k \cdot \log(2n) \cdot \log (f + 4) \cdot \mu(C)
        \end{align*}
        where the last transition follows from the facts that $C' \subseteq B'$ which implies $\mu(C') \leq \mu(B')$.
    \end{claimproof}
    Now, we apply Claim~\hyperref[claim:balanced_sep_cutoff_step]{6.1.1} to the pair $(G, Z_0)$ and obtain an $(X, \frac{95}{100})$-balanced separator $S$ in $G$ that is downward closed. Furthermore, since $\sum_{\substack{F\in\f\\ F\ni v}}x_F \geq \frac{\varepsilon}{2k}$ for every $v\in Z_0$, it follows that $\left\{x_F \cdot \frac{2k}{\varepsilon}\right\}_{F \in \f}$ is a fractional cover for $Z_0$. Consequently $\mathrm{fcov}_{\f}(Z_0) \leq \frac{2k}{\varepsilon} \cdot f$.

    Thus, 
    \begin{align*}
        \mathrm{fcov}_{\f}(S) \quad
        &\leq \quad \mathrm{fcov}_{\f}(Z_0)\ +\ 4000k \cdot \log(2n) \cdot \log (f + 4) \cdot \mu(C) \\
        &\leq \quad 1000k \cdot \log (f + 4) \cdot f\ +\ 4000k \cdot \log(2n) \cdot \log (f + 4) \cdot f \\
        &\leq \quad 5000k \cdot \log(2n) \cdot \log (f + 4) \cdot f.
    \end{align*}
    which concludes the proof of Theorem~\ref{thm:balanced_separator_rounding}.
\end{proof}

\section{Using the Dual of Balanced Separator LP}\label{sec:rounding_dual_balanced_sep}

Let $G$ be a graph, let $\f$ be a layered family of $G$, and let $\cf$ be the set of centers of elements in $\f$. Let $X \subseteq \cf$ and for each pair $u, v \in X$, let $\p_{u,v}$ denote the set of all induced paths from $u$ to $v$ in $G$. We now describe the dual of LP~\ref{lp:container_cover}, corresponding to the instance $(G, X, \mathcal{F})$. It uses non-negative real variables $\rho_u$, $\eta_{u,v}$, and $\gamma_{u,v,P}$, defined for every $u,v \in X$ and path $P \in \p_{u,v}$.
\begin{align}\label{lp:balanced_path_packing}
    \text{Maximize :}&\quad \frac{|X|}{10}\sum\limits_{u\in X} \rho_u - \sum\limits_{u,v \in X}\eta_{u,v}\\
    \nonumber \text{Subject To :}&\quad \rho_u - \eta_{u,v} - \sum\limits_{P\in\p_{u,v}}\gamma_{u,v,P} \leq 0 && \forall u,v \in X\\
    \nonumber &\quad \sum\limits_{\substack{u,v \in X\\ P\in \p_{u,v}\\ P\cap F\neq \emptyset}} \gamma_{u,v,P} \leq 1 && \forall F\in \mathcal{F}
\end{align}
It is easy to verify that the above is indeed the dual of LP~\ref{lp:container_cover}, which by strong duality~\cite{doi:10.1137/1025101} immediately implies that the linear program is feasible and has bounded optimal objective function value. Having established this, we are now ready to state the main theorem proved in this section. We remark that the general framework of this section closely follows that of Section~7 in~\cite{chudnovsky2025treewidth}. In that work, whenever the optimum value of LP~\ref{lp:balanced_path_packing} is high, one can sub-sample a set of paths that induce a subgraph of large treewidth. The main difference in this section is that we show it is possible to sample \emph{upward minimal} paths that induce a subgraph of large treewidth. The upward minimality additionally guarantees that, for every sampled path $P$ and every set $F \in \f$, the intersection $P \cap F$ has cardinality at most $2k-1$, which is a property we require for our main results.

\begin{restatable}{theorem}{BalancedSeparatorDualRounding}\label{thm:balanced_separator_dual_rounding}
    Let $G$ be a graph, $k$ be a positive integer and $\f$ be a layered family of $G$ with thickness at most $k$. Then for every subset $X \subseteq \cf$, where $\cf$ denotes the set of centers of the elements of $\f$, and positive integer $\ell$ such that $\ell\ \geq\ 7 \cdot (\,f \cdot \log n\, +\, |X|\, +\, 2\,)$, there exists an induced subgraph $H \subseteq G$ with $X \subseteq V(H)$ that satisfies the following properties:
    \begin{itemize}
        \item For every $(X,\frac{1}{2})$-balanced separator $S$ in $H$, we have that $\mathrm{cov}_{\mathcal{F}}(S) \geq f$.
        \item For every $F\in\f$, we have that $|F\cap V(H)|\leq 1 + \frac{3\, \ell}{5\, f}\cdot (2k-1)$.
    \end{itemize}
    whenever $f$, which denotes the optimum value of the dual of the balanced separator LP corresponding to $(G, X, \mathcal{F})$, is positive.
\end{restatable}

For the remainder of this section, we fix a graph $G$, the layered family $\f$ and the set $X\subseteq \cf$. We begin with the following lemma:

\begin{lemma}\label{lem:upward_minimal_is_support}
    There exists an optimal solution to LP~\ref{lp:balanced_path_packing} where, for every $u,v\in X$ and $P\in \p_{u,v}$, we have $\gamma_{u,v,P} > 0$ only if $P$ is upward minimal.
\end{lemma}
\begin{proof}
    Let $\rho_u$, $\eta_{u,v}$, and $\gamma_{u,v,P}$, for every $u,v \in X$ and $P \in \p_{u,v}$, denote the values assigned to the corresponding variables in an optimal solution to LP~\ref{lp:balanced_path_packing} that, among all optimal solutions, achieves the minimum number of nonzero $\gamma$ variables corresponding to upward non-minimal paths. If for some $u',v'\in X$, there exists an upward non-minimal $u'$--$v'$ path $P'$ with $\gamma_{u',v',P'} > 0$, then let $P_o$ be any upward minimal path that witnesses its upward non-minimality, i.e., let $P_o$ be upward minimal and satisfy $\widehat{P}_o \subsetneq \widehat{P}'$, or $\widehat{P}_o = \widehat{P}'$ and $|V(P_o)| < |V(P')|$. 
    We claim that the assignment obtained from the above by setting, for all $u, v \in X$, $\rho'_u = \rho_u$ and $\eta'_{u,v} = \eta_{u,v}$, and by modifying the $\gamma_{u,v,P}$ variables as follows:
    $$
    \gamma'_{u,v,P} = 
    \begin{cases}
    \gamma_{u,v,P_o} + \gamma_{u,v,P'}, & \text{if }u = u', v = v', P = P_o, \\
    0, & \text{if } u = u', v = v', P = P',\\
    \gamma_{u,v,P}, & \text{otherwise}, \\
    \end{cases}
    $$
    is a feasible solution. This leads to a contradiction as the new solution attains the same (optimal) objective function value but has one fewer non-zero variable corresponding to an upward non-minimal path. 

    To this end, let us consider the constraint in LP~\ref{lp:balanced_path_packing}. For ease of presentation, we refer to the constraints corresponding to vertex pairs in $X$ as \emph{type-one} constraints, and those corresponding to sets in $\f$ as \emph{type-two} constraints.
    Any constraint of type-one that does not correspond to the pair $u'$ and $v'$ contains neither $\gamma'_{u',v',P_o}$ nor $\gamma'_{u',v',P'}$ and therefore remains unchanged. The constraint of type-one corresponding to $u'$ and $v'$ contains both $\gamma'_{u',v',P_o}$ and $\gamma'_{u',v',P'}$, and since $\gamma'_{u',v',P_o} + \gamma'_{u',v',P'} = \gamma_{u',v',P_o} + \gamma_{u',v',P'}$ its left-hand side remains unchanged. Consequently, this constraint also remains satisfied, which completes the case of type-one constraints. 
    Now considering type-two, we note that any type-two constraint that does not involve $\gamma'_{u',v',P_o}$ remains satisfied, since the values of all other $\gamma'$ variables do not exceed those of $\gamma$ variables. So consider a constraint that corresponds to element $F$ in $\f$ that involves $\gamma'_{u',v',P_o}$, i.e., $P_o \cap F \neq \emptyset$. This implies that $\widehat{P}_o \cap F$ is non-empty, and consequently $\widehat{P}' \cap F$ is non-empty as well. Moreover, if $\widehat{P}' \cap F\neq\emptyset$, then by definition of upward closure there exists a vertex of $P'$ whose ancestor lies in $F$. But as elements of $\f$ are downward closed, $F$ contains a vertex of $P'$, and hence $F \cap P'$ is non-empty. Therefore, any constraint that contains $\gamma'_{u',v',P_o}$ also contains $\gamma'_{u',v',P'}$, and since $\gamma'_{u',v',P_o} + \gamma'_{u',v',P'} = \gamma_{u',v',P_o} + \gamma_{u',v',P'}$ its left-hand side is unchanged. Consequently, this constraint remains satisfied as well, and therefore the new assignment is feasible, completing the proof of the lemma by contradiction.
\end{proof}

Let us fix an optimal solution to LP~\ref{lp:balanced_path_packing}, corresponding to $(G, X, \f)$, that satisfies the conclusion of Lemma~\ref{lem:upward_minimal_is_support}. Let $\rho_u$, $\eta_{u,v}$ and $\gamma_{u,v,P}$, for every $u,v \in X$ and $P \in \p_{u,v}$, denote the values assigned to the corresponding variables in this solution. Let $f$ be the value of the objective function in this solution. For every $u, v \in X$, define $\gamma_{u,v} := \sum_{P \in \p_{u,v}} \gamma_{u,v,P}$, and let $\rho := \sum_{u \in X} \rho_u$.

\begin{lemma}\label{lem:properties_of_dual_lp}
    The optimal solution has the following properties:
    \begin{enumerate}
        \item $\sum_{v\in X}\eta_{u,v} \leq \frac{|X|}{10}\rho_u$ for every $u\in X$.
        \item $10\cdot f \leq \rho|X|$.
    \end{enumerate}
\end{lemma}

\begin{proof}
    Given any feasible solution to the LP, we can always obtain a new assignment by choosing a $u_0$ in $X$ and setting $\rho_{u_0} = 0$ and $\eta_{u_0,v} = 0$ for every $v\in X$. This preserves feasibility, since the only constraints involving these variables are of the form $\rho_{u_0} - \eta_{u_0,v} - \gamma_{u_0,v} \leq 0$, which remain satisfied even in the new assignment. Furthermore, the value of the objective function decreases by $\frac{|X|}{10}\rho_{u_0} - \sum_{v\in X}\eta_{u_0,v}$. But, if we start with the optimal solution, this modification should not strictly increase the objective function value. Therefore we get $\sum_{v\in X}\eta_{u_0,v} \leq \frac{|X|}{10}\rho_{u_0}$. Furthermore,
    $$
    10\,f \quad
    = \quad 10\left(\frac{|X|}{10} \sum_{u \in X} \rho_u - \sum_{u,v \in X} \eta_{u,v}\right) \quad
    \leq \quad 10\left(\frac{|X|}{10} \sum_{u \in X} \rho_u\right) \quad
    = \quad \rho|X|.
    $$
\end{proof}

We define $\p = \bigcup_{u,v \in X} \p_{u,v}$, and introduce a probability distribution $\mathcal{D}$ over the domain $X\times X\times (\p \cup \{\emptyset\})$. The distribution $\mathcal{D}$ is defined using the following random process which has three steps: (Step~1) Select a vertex $u \in X$ with probability $\rho_u / \rho$. (Step~2) Select a vertex $v \in X$ uniformly at random.  (Step~3) With probability $\eta_{u,v} / (\eta_{u,v} + \gamma_{u,v})$, output the triple $(u, v, \emptyset)$. Otherwise, choose a path $P \in \mathcal{P}_{u,v}$ with probability $\gamma_{u,v,P} / \gamma_{u,v}$ and output the triple $(u, v, P)$. 

We now argue that the process is well defined.
First, observe that $f > 0$ by assumption and $10\cdot f \leq \rho |X|$ by Lemma~\ref{lem:properties_of_dual_lp}, together imply $\rho > 0$. Hence $\{\, \rho_u / \rho \,\}_{u \in X}$ forms a probability distribution: each $\rho_u$ is non-negative and $\sum_{u \in X} \rho_u = \rho$. Therefore, Step~1 is well defined.
We next justify Step~3, as Step~2 is immediate.
Let $u$ be a vertex selected in Step~1. Then $\rho_u > 0$, which implies that $\eta_{u,v} + \gamma_{u,v} > 0$ for every $v \in X$, and in particular for the vertex $v$ selected in Step~2; otherwise, the constraint $\rho_u - \eta_{u,v} - \gamma_{u,v} \leq 0$ of LP~\ref{lp:balanced_path_packing} would be violated. If $\gamma_{u,v} = 0$, then necessarily $\eta_{u,v} / (\eta_{u,v} + \gamma_{u,v}) = 1$, and the process outputs $(u,v,\emptyset)$. Thus, the second case of Step~3 occurs only when $\gamma_{u,v} > 0$. In this case, $\{\, \gamma_{u,v,P} / \gamma_{u,v} \,\}_{P \in \mathcal{P}_{u,v}}$ forms a probability distribution, since each $\gamma_{u,v,P}$ is non-negative and $\sum_{P \in \mathcal{P}_{u,v}} \gamma_{u,v,P} = \gamma_{u,v}$.
Consequently, Step~3, and hence the entire process, is well defined.

% We define $\p = \bigcup_{u,v \in X} \p_{u,v}$, and introduce a probability distribution $\mathcal{D}$ over the domain $X\times X\times (\p \cup \{\emptyset\})$. The distribution $\mathcal{D}$ is defined using the following random process: Select a vertex $u \in X$ with probability $\rho_u / \rho$. Select a vertex $v \in X$ uniformly at random. With probability $\eta_{u,v} / (\eta_{u,v} + \gamma_{u,v})$, output the triple $(u, v, \emptyset)$. Otherwise, choose a path $P \in \p_{u,v}$ with probability $\gamma_{u,v,P} / \gamma_{u,v}$ and output the triple $(u, v, P)$. This process is well defined since both $\{\, \rho_u /\rho \,\}_{u \in X}$ and $\{\, \gamma_{u,v,P}/\gamma_{u,v} \,\}_{P \in \p_{u,v}}$ are probability distributions, as every $\rho_u$ and $\gamma_{u,v,P}$ take non-negative values and satisfy $\rho = \sum_{u \in X} \rho_u$ and $\gamma_{u,v} = \sum_{P\in\p_{u,v}}\gamma_{u,v,P}$. 

\begin{lemma}\label{lem:probability_bounds_1}
    $\mathcal{D}$ satisfies $\displaystyle \mathop{\mathbb{P}}\limits_{(u,v,P)\sim\mathcal{D}}\left[P = \emptyset\right] \leq \frac{1}{10}$.
\end{lemma}
\begin{proof}
    Using Lemma~\ref{lem:properties_of_dual_lp}, together with the fact that the constraint $\rho_u - \eta_{u,v} - \gamma_{u,v} \leq 0$ holds for every $u, v \in X$, we obtain:
    \begin{align*}
        \displaystyle \mathop{\mathbb{P}}\limits_{(u,v,P)\sim\mathcal{D}}\left[P = \emptyset\right] \ 
        &= \ \sum_{\substack{ u' \in X\\ \rho_{u'} > 0}} \sum_{v' \in X} \displaystyle \mathop{\mathbb{P}}\limits_{(u,v,P)\sim\mathcal{D}}\left[P = \emptyset\,\mid\, u = u',\, v = v'\right] \cdot \mathop{\mathbb{P}}\limits_{(u,v,P)\sim\mathcal{D}}\left[v = v'\right] \cdot \mathop{\mathbb{P}}\limits_{(u,v,P)\sim\mathcal{D}}\left[u = u'\right] \\
        % &= \ \sum_{u' \in X} \sum_{v' \in X} \frac{\eta_{u',v'}}{\eta_{u',v'} + \gamma_{u',v'}} \cdot \frac{1}{|X|} \cdot \frac{\rho_{u'}}{\rho} \\
        % &\leq \ \frac{1}{|X|} \sum_{u' \in X} \frac{\rho_{u'}}{\rho} \sum_{v' \in X} \frac{\eta_{u',v'}}{\rho_{u'}} \quad
        &\leq \ \frac{1}{|X|} \sum_{\substack{ u' \in X\\ \rho_{u'} > 0}} \frac{\rho_{u'}}{\rho} \sum_{v' \in X} \frac{\eta_{u',v'}}{\eta_{u',v'} + \gamma_{u',v'}} \quad
        \leq \quad \frac{1}{|X|} \sum_{\substack{ u' \in X\\ \rho_{u'} > 0}} \frac{1}{\rho} \cdot \frac{|X|}{10} \rho_{u'} \quad
        = \quad \frac{1}{10}.
    \end{align*}
\end{proof}

\begin{lemma}\label{lem:probability_bounds_2}
    If $F\in \mathcal{F}$, then $\displaystyle \mathop{\mathbb{P}}\limits_{\substack{(u,v,P)\sim\mathcal{D}}}\left[P \cap F \neq \emptyset\right] \leq \frac{1}{10\,f}$.
\end{lemma}
\begin{proof}
    Let $u',v'\in X$ and $P'\in\p_{u',v'}$. Observe that if $\rho_{u'} = 0$ or $\gamma_{u',v'} = 0$, then it follows from the definition of the process that,
    $$
    \mathop{\mathbb{P}}\limits_{\substack{(u,v,P)\sim\mathcal{D}}}\left[(u,v,P) = (u',v',P')\right] = 0.
    $$
    Hence, let us consider the case where $\rho_{u'} > 0$ and $\gamma_{u',v'} > 0$. Using the fact that the constraint $\rho_u - \eta_{u,v} - \gamma_{u,v} \leq 0$ holds for every $u, v \in X$, we obtain:
    \begin{align*}
        \displaystyle \mathop{\mathbb{P}}\limits_{\substack{(u,v,P)\sim\mathcal{D}}}\left[(u,v,P) = (u',v',P')\right] \ 
        % &= \ \mathbb{P}_{\substack{(u,v,P)\sim\mathcal{D}}}\left[P = P' \mid u = u'\right] \cdot \mathbb{P}_{\substack{(u,v,P)\sim\mathcal{D}}}\left[u = u'\right] \\
        &= \ \displaystyle \mathop{\mathbb{P}}\limits_{\substack{(u,v,P)\sim\mathcal{D}}}\left[P = P' \mid v = v', u = u'\right] \cdot \mathop{\mathbb{P}}\limits_{\substack{(u,v,P)\sim\mathcal{D}}}\left[v = v'\right] \cdot \mathop{\mathbb{P}}\limits_{\substack{(u,v,P)\sim\mathcal{D}}}\left[u = u'\right] \\
        &= \ \frac{\gamma_{u',v',P'}}{\gamma_{u',v'}} \cdot \frac{\gamma_{u',v'}}{\gamma_{u',v'} + \eta_{u',v'}} \cdot \frac{1}{|X|} \cdot \frac{\rho_{u'}}{\rho} \quad 
        % &= \ \frac{\gamma_{u',v',P'}}{\rho |X|} \cdot \frac{\rho_{u'}}{\gamma_{u',v'} + \eta_{u',v'}} \\
        \leq \quad \frac{\gamma_{u',v',P'}}{\rho |X|}.
    \end{align*}
    Now, let us fix $F \in \mathcal{F}$ and consider the probability that the sampled set $P$ has a non-empty intersection with $F$.
    \begin{align*}
        \displaystyle \mathop{\mathbb{P}}\limits_{\substack{(u,v,P)\sim\mathcal{D}}}\left[\,P \cap F \neq \emptyset\,\right] \quad
        &= \quad \sum_{u',v'\in X} \sum_{\substack{P' \in \p_{u',v'} \\ P' \cap F \neq \emptyset}} 
        \displaystyle \mathop{\mathbb{P}}\limits_{\substack{(u,v,P)\sim\mathcal{D}}}\left[\, (u,v,P) = (u',v',P')\, \right] \\
        &\leq \quad \frac{1}{\rho |X|}\sum_{\substack{u',v'\in X\\ P' \in \mathcal{P}_{u',v'} \\ \rho_{u'} > 0\\ \gamma_{u',v'} > 0 \\ P' \cap F \neq \emptyset}} \gamma_{u',v',P'} \quad 
        \leq \quad \frac{1}{\rho |X|} \quad
        \leq \quad \frac{1}{10\,f}.
    \end{align*}
    Here, the last inequality follows from Lemma~\ref{lem:properties_of_dual_lp}, and the second-to-last follows from the fact that the constraint, $\sum_{\substack{u,v \in X\\ P\in \p_{u,v}\\ P\cap F\neq \emptyset}} \gamma_{u,v,P} \leq 1$, holds for every $F \in \mathcal{F}$.
\end{proof}

Finally, we require the following lemma, which relates balanced separators of the set $A$ to $A_1$--$A_2$ separators for a specific partition $A_1$, $A_2$ of $A$.

\begin{lemma}\label{lem:separation_vs_separators}
    Let $G$ be a graph and $A\subseteq V(G)$ be a vertex subset. If $G$ has an $(A,\frac{1}{2})$-balanced separator $S$, then there exists a partition $A_1$, $A_2$ of $A$ such that $\max\{|A_1|, |A_2|\} \leq \frac{2|A|}{3}$ and $S$ is an $A_1$--$A_2$ separator.
\end{lemma}
\begin{proof}
    If $S\cap A$ contains at least $\frac{|A|}{2}$ elements, then selecting any subset $A_1 \subseteq S\cap A$ of size $\frac{|A|}{2}$ and setting $A_2 := A \setminus A_1$ suffices to prove the lemma. Otherwise, let $C_1, \dots, C_r$ be connected components of $G-S$ that have a non-empty intersection with $A$. Let $A'_i := A \cap C_i$ for every $i \in [r]$, and let $A'_{r+1} := A \cap S$. By relabeling the sets if necessary, assume that $|A'_i| \geq |A'_{i+1}|$ for all $i \in [r]$. Let $q$ be the smallest integer in $[r+1]$ such that $\sum_{i=1}^{q}|A'_i| \geq \frac{|A|}{3}$. We claim that $\sum_{i=1}^{q}|A'_i| \leq \frac{2|A|}{3}$. If $q = 1$, the claim holds trivially, as $S$ is an $(A,\frac{1}{2})$-balanced separator and consequently $|A'_1|\leq \frac{|A|}{2}$. Otherwise, observe that
    $$
    \sum_{i=1}^{q}|A'_i| \quad
    = \quad |A'_{q}|\ +\ \sum_{i=1}^{q-1}|A'_i| \quad
    \leq \quad |A'_{q-1}|\ +\ \sum_{i=1}^{q-1}|A'_i| \quad
    \leq \quad 2\sum_{i=1}^{q-1}|A'_i| \quad
    \leq \quad \frac{2|A|}{3}.
    $$
    Define $A_1 := \cup_{i=1}^{q} A'_i$ and $A_2 := A\setminus A_1$. Since the sets in $\{A'_i\}_{i\in [r+1]}$ are pairwise disjoint, we have that $|A_1| = \sum_{i=1}^{q}|A'_i|$. Hence, $|A_1| \leq \frac{2|A|}{3}$ and since $|A_1| \geq \frac{|A|}{3}$, we have $|A_2| \leq \frac{2|A|}{3}$. Furthermore, it follows from our construction of $A_1$, $A_2$ that $S$ is an $A_1$--$A_2$ separator, which concludes the proof of the lemma.
\end{proof}

Now we are ready to prove Theorem~\ref{thm:balanced_separator_dual_rounding}.

\BalancedSeparatorDualRounding*
\begin{proof}
    Let $\{(u_i, v_i, P_i)\}_{i\in[\,\ell\,]}$ be $\ell$ independent samples drawn from the distribution $\mathcal{D}$ over $X\times X\times (\p \cup \{\emptyset\})$. Let $H$ be the subgraph of $G$ induced by the union of all sampled sets and $X$, that is, $H = G[X\cup\bigcup_{i \in [\ell]} P_i]$. We claim that $H$ has the desired properties with good probability.

    \begin{claim}\label{claim:no_balanced_separator}
        With probability at least $\frac{3}{4}$, the following holds: for every subfamily $\mathcal{F}' \subseteq \mathcal{F}$ of size at most $f$, and for every partition $X_1, X_2$ of $X$ such that $\max\{|X_1|, |X_2|\} \leq \frac{2|X|}{3}$, there exists an $X_1$--$X_2$ path in $H$ that does not intersect any set in $\mathcal{F}'$. 
    \end{claim}
    \begin{claimproof}
        We define an {\em eligible triple} to be a triple $(\mathcal{F}', X_1, X_2)$ such that $\mathcal{F}' \subseteq \mathcal{F}$ is a subfamily of size at most $f$, and $X_1$, $X_2$ is a partition of $X$ satisfying $\max\{|X_1|, |X_2|\} \leq \frac{2|X|}{3}$. An eligible triple $(\mathcal{F}', X_1, X_2)$ is {\em bad} if every $X_1$--$X_2$ path in $H$ intersects at least one set in $\mathcal{F}'$. Let $\chi$ be the event that the conclusion of the claim is false, namely that there exists an eligible bad triple. To prove the claim it suffices to show $\mathbb{P}[\chi] \leq \frac{1}{4}$.
        
        For every eligible triple  $(\mathcal{F}', X_1, X_2)$, we define $\chi(\mathcal{F}', X_1, X_2)$ to be the event that this triple is bad. We will prove that for every eligible triple $(\mathcal{F}', X_1, X_2)$ we have $\mathbb{P}[\chi(\mathcal{F}', X_1, X_2)] \leq (\frac{9}{10})^\ell$, then $\mathbb{P}[\chi] \leq \frac{1}{4}$ follows by a simple union bound over all eligible triples. More concretely we have the following.
        $$
            \mathbb{P}\left[\chi\right] \quad \leq \quad \left(\frac{9}{10}\right)^{\ell} \cdot 2^{|X|} \cdot \binom{|\mathcal{F}|}{|\mathcal{F}'|} \quad
            \leq \quad 2^{|X|\ +\ |\mathcal{F}'| \log |\mathcal{F}|\ -\ \ell \log\left(\frac{10}{9}\right)} \quad
            \leq \quad \frac{1}{4}.
        $$
                
        To prove that $\mathbb{P}[\chi(\mathcal{F}', X_1, X_2)] \leq (\frac{9}{10})^\ell$ we observe that the event $\chi(\mathcal{F}', X_1, X_2)$ does not occur if there exists an $i \in [\ell]$ such that $P_i$ is an $X_1$--$X_2$ path in $H$ that is disjoint from every set in $\mathcal{F}'$. For every eligible triple $(\mathcal{F}', X_1, X_2)$ and every $i \in [\ell]$ we define the event $\overline{\chi}(\mathcal{F}', X_1, X_2, i)$ that $P_i$ is an $X_1$--$X_2$ path in $H$ that is disjoint from every set in $\mathcal{F}'$. For every eligible triple $(\mathcal{F}', X_1, X_2)$ the events in $\{\overline{\chi}(\mathcal{F}', X_1, X_2, i) ~|~ i \in [\ell]\}$ are independent. Thus, to prove  $\mathbb{P}[\chi(\mathcal{F}', X_1, X_2)] \leq (\frac{9}{10})^\ell$ it suffices to show that $\mathbb{P}[\overline{\chi}(\mathcal{F}', X_1, X_2, i)] \geq \frac{1}{10}$ for every $i \in [\ell]$.
        
        To lower bound $\mathbb{P}[\overline{\chi}(\mathcal{F}', X_1, X_2, i)]$ we observe that $(u_i, v_i, P_i)$ is sampled according to the distribution $\mathcal{D}$. The event $\overline{\chi}(\mathcal{F}', X_1, X_2, i)$ occurs unless $|\{u_i, v_i\} \cap X_1| \neq 1$, or $P_i = \emptyset$, or there exists $F \in \mathcal{F}'$ such that $P_i \cap F \neq \emptyset$. Since $v_i$ is sampled uniformly from $X$ we have that  $\mathbb{P}[v_i \in X_1 | u_i \in X_1] \leq \frac{2}{3}$ and $\mathbb{P}[v_i \in X_2 | u_i \in X_2] \leq \frac{2}{3}$. The law of conditional probability applied to the event $u_i \in X_1$ now yields $\mathbb{P}[|\{u_i, v_i\} \cap X_1| \neq 1] \leq \frac{2}{3}$. By Lemma~\ref{lem:probability_bounds_1} we have that $\mathbb{P}[P_i = \emptyset] \leq \frac{1}{10}$. For every $F \in \mathcal{F}'$, by Lemma~\ref{lem:probability_bounds_2} we have that $\mathbb{P}[P_i \cap F \neq \emptyset] \leq \frac{1}{10\,f}$. A union bound over all $F \in \mathcal{F}'$ yields that the probability that there exists $F \in \mathcal{F}'$ such that $P_i \cap F \neq \emptyset$ is at most $\frac{|\mathcal{F}'|}{10\,f} \leq \frac{1}{10}$. We conclude that $\mathbb{P}[\overline{\chi}(\mathcal{F}', X_1, X_2, i)] \geq 1 - \frac{2}{3} - \frac{1}{10} - \frac{1}{10} \geq \frac{1}{10}$, giving the desired lower bound on $\mathbb{P}[\overline{\chi}(\mathcal{F}', X_1, X_2, i)]$, and completing the proof of the claim. 
    \end{claimproof}

    \begin{claim}\label{claim:max_deg}
        With probability at least $\frac{3}{4}$, the following holds: For every $F\in\f$, we have that $|F\cap V(H)|\leq 1 + \frac{3\, \ell}{5\, f}\cdot (2k-1)$.
    \end{claim}
    \begin{claimproof}
        Let $F \in \mathcal{F}$, and let $\chi_F$ denote the number of sets $P_i$, out of a total of $\ell$ samples, that intersect $F$. We define $\chi_{\mathcal{F}} := \max\{\chi_F\, |\, F\in\mathcal{F}\}$ and show that with probability at least $\frac{3}{4}$, the event $\chi_{\mathcal{F}} < 6\cdot \frac{\ell}{10\,f}$ occurs. 
        
        Let $F\in\mathcal{F}$, since the probability that a fixed set $P_i$ intersects $F$ is at most $\frac{1}{10\, f}$, it follows from linearity of expectations that the expected value of $\chi_F$ is at most $\frac{\ell}{10\, f}$. Applying union bound over all $F\in\mathcal{F}$, the Chernoff bound from Theorem \ref{thm:chernoff} to $\chi_F$ for every $F\in\mathcal{F}$, and using the lower bound on $\ell$, we get: 
        $$
        \mathbb{P}\left[\chi_{\mathcal{F}} < 6 \cdot \frac{\ell}{10\,f}\right] \quad
        \geq \quad 1\ -\ \sum_{F\in\mathcal{F}}\mathbb{P}\left[\chi_F \geq 6 \cdot \frac{\ell}{10\,f}\right] \quad
        \geq \quad 1\ -\ |\mathcal{F}| \cdot 2^{-6 \cdot \frac{\ell}{10\,f}} \quad
        \geq \quad \frac{3}{4}.
        $$
        Suppose that in our sampled subgraph $H$, the event $\chi_{\mathcal{F}} < \frac{3\, \ell}{5\,f}$ occurs. Since $\mathcal{D}$ only assigns non-zero probability to upward minimal paths, every sampled path is upward minimal. Hence, by Lemma~\ref{lem:minimal_paths_are_short}, each path intersects a set in $\f$ at most $2k-1$ times. Furthermore, since every set contains at most one element of $X$ by definition, we get $|F\cap V(H)|\leq 1 + \frac{3\, \ell}{5\, f}\cdot (2k-1)$.
    \end{claimproof}

Since the probability that the sampled subgraph satisfies both the properties stated in Claim~\ref{claim:max_deg} and Claim~\ref{claim:no_balanced_separator} is at least $\frac{1}{2}$, there exists an induced subgraph $H \subseteq G$ that satisfies both. Let $H$ be such an induced subgraph. Combining Lemma~\ref{lem:separation_vs_separators} with Claim~\ref{claim:no_balanced_separator}, we conclude that $H$ has no $(X,\frac{1}{2})$-balanced separator $S$ with $\mathrm{cov}_{\mathcal{F}}(S) \leq f$, thereby completing the proof of Theorem~\ref{thm:balanced_separator_dual_rounding}.
\end{proof}

\section{Some Results for an Auxiliary Class}\label{sec:weird_class}

Let $\mathcal{C}^c$ denote a family of graphs that is closed under the induced subgraph relation, with the additional property that every graph $G$ in $\mathcal{C}^c$ satisfies $\tw(G) \leq (\Delta(G))^c$. Let $\mathcal{C}^c_{d}$ be a subfamily of $\mathcal{C}^c$ which contains all graphs in $\mathcal{C}^c$ of degeneracy at most $d$. We begin with the following definition. Given a graph $G$ and a positive integer $d$, a layered family is \emph{$d$-neighborhood-witnessing} if  for every vertex $v\in V(G)$, every $F\in \f$ that contains $v$ satisfies $|N[v]\setminus F| \leq d$.  
% \todo[inline] {How about this:}
% Seconded
The next lemma can be easily deduced from Theorem~7.1 of Abrishami, Chudnovsky, Hajebi and Spirkl~\cite{tw3}, but we include a proof here for completeness.
%The proof of the next lemma closely resembles that of a theorem of Abrishami, Chudnovsky, Hajebi and Spirkl~\cite{tw3}.
% \todo[inline]{A lot of the next lemma is re-proving a theorem from \cite{tw3}. We should at least cite it.}
% \todo[inline]{(Ajay) Does the added line work? I can also add the theorem number if you can help me find it.}

\begin{lemma}\label{lem:degeneracy_implies_layered_family}
    Let $d$ be a positive integer and $G$ be a graph of degeneracy $d$. Then there exists a $4d$-neighborhood-witnessing layered family $\f$ of $G$ with thickness at most $\log (4n)$.
\end{lemma}
\begin{proof}
    We iteratively construct an ordered partition of $V(G)$ as follows. Starting with $i=1$, define $V_i := V(G) \setminus \bigcup_{j=1}^{i-1} U_j$ and let $U_i$ be the set of vertices in the induced subgraph $G[V_i]$ whose degree in $G[V_i]$ is at most $4d$. Let $k$ be the last index for which $V_k$ is non-empty. Let $\Pi := (U_1, U_2, \dots, U_k)$. Observe that $\Pi$ is an ordered partition of $V(G)$ and let $\f$ be the layered family corresponding to $(G,\Pi)$.

    First, we claim that $k$ is at most $\log (4n)$ which directly implies that $\f$ can have thickness at most $\log (4n)$. For this claim, observe that any graph of degeneracy $d$ has average degree at most $2d$, and consequently, at least half of its vertices have degree at most $4d$. Therefore, for every $i\in [k]$, the set $U_i$ contains at least half of the vertices of $V_i$. Since $|V_1| = n$ and $|V_i| \leq \lfloor |V_{i-1}|/2 \rfloor$ for all $i \geq 2$, it follows that $|V_{\lceil \log n \rceil + 2}| = 0$ and by definition $k$ is at most $\log (4n)$. 
    For the other property, let $v$ be a vertex in $G$, and $F \in \f$ be such that $v\in F$. Now, let $U_i$ be the part in $\Pi$ containing $v$, and observe that by definition, $v$ has at most $4d$ neighbors in $V_i = \bigcup_{j \geq i} U_j$. Moreover, every neighbor of $v$ in $V(G) \setminus V_i = \bigcup_{j < i} U_j$ is a descendant of $v$ in the layer graph of $(G, \Pi)$. Therefore every $F\in\f$ that contains $v$ satisfies $|N[v] \setminus F| \leq 4d$, which completes the proof of the lemma.
\end{proof}

We note the following propositions, which we need for the proof of the main theorems.

\begin{proposition}[\cite{cygan2015parameterized}]\label{prop:tw_and_balsep}
    Let $G$ be a graph of treewidth at most $k$, and let $X\subseteq V(G)$ be a vertex subset. Then there exists an $(X,\frac{1}{2})$-balanced separator $S$ in $G$ of size at most $k + 1$.
\end{proposition}

\begin{proposition}[\cite{chvatal1979greedy}]\label{prop:packing_covering_LP}
    Let $U$ be a set and let $\f$ be a family of subsets of $U$. If $S\subseteq U$ satisfies $\mathrm{fcov}_{\f}(S)\, \leq\, k$, then it holds that $\mathrm{cov}_{\f}(S)\, \leq\, k\cdot \log|U|$.
\end{proposition}

With this, we are prepared to state and prove a key lemma which we require for the proof of Theorem~\ref{thm:F_treewidth}.

\begin{lemma}\label{lem:F_balanced_separator}
    Let $c,d$ be positive integers and $G$ be a graph in $\mathcal{C}^c_d$. Let $k$ be a positive integer and $\f$ be a $4d$-neighborhood-witnessing layered family of $G$ with thickness at most $k$. Let $\cf$ be the set of centers of elements in $\f$, and $X\subseteq\cf$ be of size at most 
    $$
    110000k\cdot \log^2(2n)\cdot f\log (4f);\quad 
    \text{ where }
    f = (1000000d\cdot k^2\cdot \log^3 (2n))^{c}.
    $$  
    Then there exists a set $S\subseteq\cf$ such that $\widecheck{S} = \bigcup_{w \in S} F_w$ is an $(X,\frac{95}{100})$-balanced separator in $G$ and $|S|\, \leq\, 5000k\cdot \log^2(2n)\cdot f\log (4f)$.
\end{lemma}
\begin{proof}
    Note that if $f\geq n$, then the lemma holds trivially by defining $S := \cf$, so let us assume that this is not the case. Consider the optimum value $f'$ of LP~\ref{lp:container_cover} corresponding to $(G,X,\f)$. 
    If $f'\geq f$, then let $\ell := 7\, (\,f'\cdot\log n\, +\, 110000k\cdot \log^2(2n)\cdot f\log (4f)\, +\, 2\,)$ and apply Theorem~\ref{thm:balanced_separator_dual_rounding} on $(G,\f,X,\ell)$ to obtain an induced subgraph $H \subseteq G$. Let $v$ be a vertex of $H$ and $F$ be an element of $\f$ that contains $v$. 
    By the properties of $\f$, we have $|N[v] \setminus F| \leq 4d$. Moreover, since $|F \cap V(H)|$ is upper bounded by the properties of $H$ asserted by Theorem~\ref{thm:balanced_separator_dual_rounding}, it follows that
    \begin{align*}
    |N[v] \cap V(H)| \quad 
    &\leq \quad |(N[v] \cap V(H))\cap F| +\ |(N[v] \cap V(H))\setminus F|\\
    &\leq \quad 1\ +\ \frac{3\, \ell}{5\, f'}\cdot (2k - 1)\ +\ 4d\\
    &\leq \quad 1\ +\ \frac{21}{5}\cdot 2k\cdot (\,\log n\, +\, 110000k\cdot \log^2(2n)\cdot \log (4f)\, +\, 2)\ +\ 4d\\
    &< \quad (\, 1\, +\, \frac{42}{5}\, +\, \frac{42}{5}\cdot 110000\, +\, \frac{84}{5}\, +\, 4\, ) \cdot d\cdot k^2\cdot \log^2(2n)\cdot \log (4f)\\
    &< \quad 1000000d \cdot k^2 \cdot \log^3 (2n).
    \end{align*}
    Therefore, maximum degree of $H$ is less than $1000000d \cdot k^2 \cdot \log^3 (2n)$. 
    Meanwhile, every $(X,\frac{1}{2})$-balanced separator $S$ in $H$ satisfies $\mathrm{cov}_{\f}(S) \geq f$, which implies that $|S| \geq f$. Applying the contrapositive of Proposition~\ref{prop:tw_and_balsep} to $H$, we conclude that $\tw(H) \geq (1000000d\cdot k^2\cdot \log^3 (2n))^{c}$. 
    But, these observations contradict the fact that since $H$ is an induced subgraph of $G$, it must also belong to the family $\mathcal{C}^c_d$. 
    
    Hence $f'$ is at most $f$, and consequently by Theorem~\ref{thm:balanced_separator_rounding}, there exists an $(X,\frac{95}{100})$-balanced separator $S'$ in $G$ satisfying $\mathrm{fcov}_{\f}(S')\, \leq\, 5000k\cdot \log(2n)\cdot f\log(f+4)\, \leq\, 5000k\cdot \log(2n)\cdot f\log (4f)$. Proposition~\ref{prop:packing_covering_LP} implies that $\mathrm{cov}_{\f}(S')\, \leq\, 5000k\cdot \log^2(2n)\cdot f\log (4f)$. 
    Thus, choosing any $S \subseteq \cf$ with $|S| \leq 5000k \cdot \log^2(2n) \cdot f \log (4f)$ such that $\f_S$ covers $S'$ completes the proof of the lemma.
\end{proof}

\begin{theorem}\label{thm:F_treewidth}
    Let $c,d$ be positive integers and $G$ be a graph in $\mathcal{C}^c_d$. Let $k$ be a positive integer and $\f$ be a $4d$-neighborhood-witnessing layered family of $G$ with thickness at most $k$. Let $\cf$ be the set of centers of elements in $\f$, and $X$ be a subset of $\cf$. Then $G$ has a rooted tree decomposition $(T,\chi)$, with root $r\in V(T)$, such that $\chi(r) \supseteq \bigcup_{w \in X} F_w$ and for every $x\in V(T)$, there exists a subset $X'\subseteq \cf$ such that,
    $$
    |X'| \quad
    \leq \quad \max\{\,|X|,\, 110000k\cdot \log^2 (2|V(G)|)\cdot f\log (4f)\,\}\, +\, 5000k\cdot \log^2 (2|V(G)|)\cdot f\log (4f),
    $$
    where $f = (1000000d\cdot k^2\cdot \log^3 (4|V(G)|))^{c}$; and the bag satisfies $\chi(x) = \bigcup_{w \in X'} F_w$.
\end{theorem}
\begin{proof}
    We prove the theorem using induction on the size of $V(G)$. For the base case, when $G$ has a single vertex, the theorem holds trivially. 
    Assume inductively that the statement holds for all graphs with strictly fewer than $n$ vertices, and consider a graph $G$ on exactly $n$ vertices. Note that if
    $$
    |\cf| \quad
    \leq \quad \max\{\,|X|,\, 110000k \cdot \log^2(2n) \cdot f \log (4f)\,\}
    + 5000k \cdot \log^2(2n) \cdot f \log (4f),
    $$
    then the theorem holds trivially, since a rooted tree decomposition with one node containing all of $V(G)$ satisfies the requirements. Hence, assume that this is not the case. Without loss of generality, also assume that $|X| \geq \lfloor 110000k \cdot \log^2(2n) \cdot f \log (4f) \rfloor$. Indeed, if this were not the case, we may add arbitrary elements of $\cf$ to $X$ until equality holds, since proving the theorem for the resulting larger set immediately implies the statement for the original set.
    Now, if $|X| > 110000k \cdot \log^2(2n) \cdot f \log (4f)$, define $Y \subseteq X$ to be any subset of size $\lfloor 110000k \cdot \log^2(2n) \cdot f \log (4f) \rfloor$; otherwise, let $Y := X$. Let $S \subseteq \cf$ be the set obtained by applying Lemma~\ref{lem:F_balanced_separator} to the instance $(G,\f,Y)$. Recall that $\widecheck{S} = \bigcup_{w \in S} F_w$, and let $C$ be a connected component of $G - \widecheck{S}$. Define $G_C := G[C \cup \widecheck{S}]$, $\f_C := \{\, F_w \in \f \mid w \in S \cup (\cf \cap C) \,\}$, and $X_C := S \cup (X \cap C)$. We wish to apply the inductive assumption to the instance $(G_C, \f_C, X_C)$. 
    
    To justify this step, we need to establish two facts. 
    First, let $C$ be a connected component of $G - \widecheck{S}$. Since $\widecheck{S}$ is a $(Y,\frac{95}{100})$-balanced separator, $C$ contains at most $\frac{95}{100}\,|Y|$ vertices of $Y$. Since $|Y| = \lfloor 110000k \cdot \log^2(2n) \cdot f \log (4f) \rfloor$, at least $5500k \cdot \log^2(2n) \cdot f \log (4f)$ vertices of $Y$ (and consequently of $X$ since $Y \subseteq X$) must lie outside $C$. Finally, as $|S| \leq 5000k \cdot \log^2(2n) \cdot f \log (4f)$, at least $500k \cdot \log^2(2n) \cdot f \log (4f)$ vertices of $X$ must lie outside $G_C$. Thus $G_C$ has strictly fewer vertices than $G$, for every connected component $C$ of $G - \widecheck{S}$. 
    Second, by Lemma~\ref{lem:component_is_downward_closure}, $C \cup \widecheck{S}$ is the downward closure of ${S \cup (\cf \cap C)}$. Hence by Lemma~\ref{lem:layered_family_subfamily}, $\f_C$ is a layered family of $G_C$ with thickness at most $k$.

    Therefore it is safe to apply the theorem inductively to the instance $(G_C, \f'_C, X_C)$. Let $(T_C,\chi_C)$ denote the resulting rooted tree decomposition of $G_C$. We now define a rooted tree decomposition $(T,\chi)$ of $G$ as follows. For the tree $T$ we introduce a new root node $r$, and for each connected component $C$ of $G - \widecheck{S}$, connect $r$ to the root of $T_C$. We define a bag function $\chi$ on $V(T)$ by setting $\chi(r) := \bigcup_{w \in X \cup S} F_w,$ and for every other node $x \in V(T)$, letting $\chi(x) := \chi_C(x)$, where $C$ is the connected component such that $x\in V(T_C)$. 
    A straightforward verification of the definitions shows that the resulting pair $(T,\chi)$ is indeed a rooted tree decomposition of $G$. Furthermore, the construction satisfies $\chi(r) \supseteq \bigcup_{w \in X} F_w$. Now, let $x \in V(T)$. If $x = r$, then $\chi(x) = \bigcup_{w \in X \cup S} F_w$, and 
    $$
    |X \cup S|\quad \leq\quad |X|\, +\, 5000k \cdot \log^2(2|V(G)|) \cdot f \log (4f).
    $$
    Otherwise, $x \in V(T_C)$ for some component $C$ of $G - \widecheck{S}$. Let $f_C := (1000000d\cdot k^2 \log^3 (2|C \cup \widecheck{S}|))^{c}$ and observe that by induction, there exists a set $X'_C \subseteq \mathcal{C}_{\f_C}$ such that $\chi(x) = \bigcup_{w \in X'_C} F_w$ and
    \begin{align*}
        |X'_C| \quad 
        &\leq \quad \max\Bigl\{ |X_C|,\ 110000k \cdot \log^2 \bigl(2|C \cup \widecheck{S}|\bigr)\cdot f_C \log (4f_C) \Bigr\} \\
        &\qquad +\, 5000k \cdot \log^2 \bigl(2|C \cup \widecheck{S}|\bigr)\cdot f_C \log (4f_C) \\
        &\leq \quad \max\Bigl\{ \bigl(|X| - \frac{5}{100}|Y| + 5000k \cdot \log^2(2n)\cdot f \log (4f)\bigr),\ 110000k \cdot \log^2(2n)\cdot f \log (4f)\Bigr\} \\
        &\qquad +\, 5000k \cdot \log^2(2n)\cdot f \log (4f) \\
        &\leq \quad \max\Bigl\{ |X|,\ 110000k \cdot \log^2(2n)\cdot f \log (4f) \Bigr\} \\
        &\qquad +\, 5000k \cdot \log^2(2n)\cdot f \log (4f) .
    \end{align*}
    Since $\mathcal{C}_{\f_C} \subseteq \cf$, the same set $X'_C$ witnesses the required property for the node $x$ as an element of $V(T)$. Consequently, $(T,\chi)$ is a rooted tree decomposition satisfying all the desired properties, which completes the proof of the theorem.
\end{proof}

Recall that for positive integers $k$ and $r$, a subset $S \subseteq V(G)$ is said to be $(k,r)$--coverable if there exists a set $C \subseteq V(G)$ with $|C| \le k$ such that every vertex of $S$ has a path on at most $r$ vertices to some vertex of $C$. Let $c,d$ be positive integers, let $G \in \mathcal{C}^c_d$ be an $n$-vertex graph, and let $\f$ be a $4d$-neighborhood-witnessing layered family in $G$ with thickness at most $\log (4n)$. We apply the above theorem with $X = \emptyset$. Observe that if $F \in \f$ has center $w$, then $F$ is contained in the ball of radius $\log (4n)$ centered at $w$. Combining this observation with the theorem immediately yields the following.

\begin{corollary}\label{thm:coarse_treewidth_aux_class}
    Let $c,d$ be positive integers and let $G$ be a graph in $\mathcal{C}^c_d$. Then $G$ has a tree decomposition $(T,\chi)$ such that for every $x\in V(T)$, the bag $\chi(x)$ is $(k, r)$--coverable in $G$, where $k$ and $r$ satisfy:
    \begin{align*}
        r\quad &\leq\quad \log (4n).\\
        k\quad &\leq\quad 115000 \cdot \log^3 (4n) \cdot f \log (4f);\quad \text{ where } f := (1000000d\cdot \log^5 (4n))^{c}.
    \end{align*}    
\end{corollary}

% \todo[inline]{The corollary also follows from \cite{abrishami2025coarsetreedecompositionscoarse}; let us mention that}
%  Done

Here we remark that the above corollary also follows from Lemma~\ref{lem:F_balanced_separator} combined with a theorem of Abrishami, Czy\.{z}ewska, Kluk, Pilipczuk, Pilipczuk, and Rz\k{a}\.{z}ewski~\cite{abrishami2025coarsetreedecompositionscoarse}.
We also prove the following theorem regarding graphs of bounded degeneracy which we need towards the proof of Theorem~\ref{thm:coarse_induced_menger_our_class}.

\begin{theorem}\label{thm:coarse_induced_menger_aux_class}
    Let $d, f$ be positive integers and let $G$ be an $n$-vertex graph of degeneracy $d$. Let $\f$ be a $4d$-neighborhood-witnessing layered family of $G$ with thickness at most $\lceil \log (2n) \rceil$, and let $A, B \subseteq V(G)$ be vertex subsets. Then, one of the following exists:
    \begin{itemize}
        \item An $A$--$B$ separator $S\subseteq V(G)$ in $G$, such that $\mathrm{fcov}_{\f}(S)\, \leq\, 8 \log^2 (4n) \cdot f$. 
        \item An induced subgraph $H$ of $G$, of maximum degree at most $24 \log^2 (4n) + 4d$, such that every $A_H$--$B_H$ separator $S$ in $H$ must have $\mathrm{cov}_{\f}(S)\, \geq\, \frac{f}{12}$ in $G$, where $A_H := A\cap V(H)$ and $B_H := B\cap V(H)$.
    \end{itemize}
\end{theorem}
\begin{proof}
    Consider the optimum solution value $f'$ of the LP~\ref{lp:ab_separator_using_f} corresponding to $(G,\f,A,B)$. If $f' \leq f$, then applying Theorem~\ref{thm:ab_separator} on $(G,\f,A,B)$ yields an $A$--$B$ separator $S\subseteq V(G)$ in $G$, such that $\mathrm{fcov}_{\f}(S)\ \leq\ 8\log (4n) \cdot \log(2n) \cdot f' \leq 8 \log^2 (4n) \cdot f$. Otherwise, we apply Theorem~\ref{thm:ab_path_packing} to the instance $(G,\f,A,B)$ with $\ell := \log(2n)$, obtaining a multiset $\q$ of $A$--$B$ paths. 
    Let $H$ be the subgraph of $G$ induced by the vertices that lie on paths in $\q$. We claim that $H$ has the desired properties. First, for every $F \in \f$, at most $6 (\log(2n) + 1) \leq 12\log(2n)$ paths of $\q$ intersect $F$, while $\q$ contains a total of $f \cdot \log(2n)$ paths. Therefore, any $A_H$--$B_H$ separator $S$ in $H$ must satisfy $\mathrm{cov}_{\f}(S) \geq \frac{f}{12}$. Second, let $v$ be a vertex of $H$ and let $F \in \f$. Since at most $12 \log(2n)$ paths of $\q$ intersect $F$, and for every path $P \in \q$ we have $|F \cap P|\, \leq\, 2\log (4n) - 1$, it follows that $|F \cap V(H)|\, \leq\, 24 \log^2 (4n)$. As $\f$ is $4d$-neighborhood-witnessing, we conclude that $|N[v] \cap V(H)|\, \leq\, |F \cap V(H)| + 4d\, \leq\, 24 \log^2 (4n) + 4d$, which completes the proof of the theorem.
\end{proof}

% -----------------------------------------------------------------------------------------------------------------

% Uncomment the below two lines if we are not splitting into sections

% \section{Proof of Theorems for $K_{t,t}$-induced-minor-free Graphs}
% \input{ZZ_Section 9: Proof for our class}

% -----------------------------------------------------------------------------------------------------------------

% Uncomment below eight lines if we are splitting into sections

\section{Vertex partitions of graphs with no $K_{t,t}$-induced-minors}

Let $G$ be a graph, $r$ a positive integer, and $S \subseteq V(G)$ a vertex subset. Recall that $S$ is \emph{$(1,r)$--coverable} if there exists a vertex $v\in V(G)$ such that every vertex of $S$ has a path on at most $r$ vertices to $v$. A partition $(X_1, \dots, X_\ell)$ of $S$ is called an \emph{$r$–radius vertex partition} of $S$ in $G$ if for every $i \in [\ell]$ it holds that every connected component $C$ of $G[X_i]$ is $(1,r)$--coverable in $G$. Notice that if $u,v\in C$ for some $i\in[\ell]$ and some connected component $C$ of $G[X_i]$, then there exists a $u$--$v$ path on at most $2r$ vertices in $G$. We define $\pi_r(S,G)$ as the smallest integer $\ell$ for which there exists an $r$–radius vertex partition of $S$ in $G$ into $\ell$ parts. We simply write $\pi_r(G)$ to denote $\pi_r(V(G),G)$.
We begin with the following lemma.

\begin{lemma}\label{lem:properties_of_pi}
    Let $G$ be a graph and $r$ be a positive integer. Then, for all $A, B \subseteq V(G)$,
    \[
        \max\bigl\{\pi_r(A,G),\,\pi_r(B,G)\bigr\} \
        \leq \ \pi_r(A \cup B,\,G) \
        \leq \ \pi_r(A,G) + \pi_r(B,G).
    \]
    Furthermore, the first inequality holds with equality whenever $A$ and $B$ are anticomplete in $G$.
\end{lemma}
\begin{proof}
    Set $p := \pi_r(A,G)$, $q := \pi_r(B,G)$, and let $(X_1,\dots,X_p)$, $(Y_1,\dots,Y_q)$ be optimal $r$-radius vertex partitions of $A$ and $B$ in $G$.

    For the first inequality, let $(W_1,\dots,W_\ell)$ be an optimal $r$-radius vertex partition of $A \cup B$ in $G$. For each $j \in [\ell]$, every connected component of $G[W_j \cap A]$ is a subgraph of a connected component of $G[W_j]$, and hence $(1,r)$-coverable. Thus $(W_1 \cap A, \dots, W_\ell \cap A)$ is an $r$-radius vertex partition of $A$, giving $\pi_r(A,G) \leq \ell$. By a symmetric argument, $\pi_r(B,G) \leq \ell$, and therefore $\max\{\pi_r(A,G),\pi_r(B,G)\} \leq \ell = \pi_r(A \cup B, G)$.

    For the second inequality, define $Z_j := X_j$ for $j \in [p]$ and $Z_{p+j} := Y_j \setminus A$ for $j \in [q]$, and observe that $(Z_1,\dots,Z_{p+q})$ is a partition of $A \cup B$ into $p + q$ parts. Furthermore every connected component $C$ of $G[Z_j]$ is a subgraph of a connected component $C'$ of either $G[X_j]$ or $G[Y_{j-p}]$. Since $C'$ is $(1,r)$-coverable in $G$, the same center vertex witnesses $(1,r)$-coverability of $C$. It follows that $(Z_1,\dots,Z_{p+q})$ is an $r$-radius vertex partition of $A \cup B$ in $G$, giving $\pi_r(A \cup B, G) \leq p + q$.
    
    Now, suppose $A$ and $B$ are anticomplete. Set $\ell := \max\{p,q\}$ and, by appending empty parts if necessary, assume both partitions have exactly $\ell$ parts. Define $Z_j := X_j \cup Y_j$ for $j \in [\ell]$. Since $X_j \subseteq A$ and $Y_j \subseteq B$ are anticomplete, every connected component of $G[Z_j]$ lies entirely within $G[X_j]$ or entirely within $G[Y_j]$, and is therefore $(1,r)$-coverable. Hence $\pi_r(A \cup B, G) \leq \ell = \max\{p,q\}$, which completes the proof of the lemma.
\end{proof}

All of our main theorems rely on the following property of graphs that exclude $K_{t,t}$ as an induced minor:

\begin{restatable}{theorem}{KttFreeimpliesBoundedRadiusPartition}\label{thm:ktt_free_implies_bounded_radius_partition}
   Let $G$ be a $K_{t,t}$-induced-minor-free graph for a positive integer $t$. Then $\pi_4(G) \leq t \cdot 2^t$.
\end{restatable}

For this, we need the following lemma.

\begin{lemma}\label{lem:obtain_new_layer}
    Let $G$ be a graph and let $Y \subseteq V(G)$ be such that $G[Y]$ is connected. Then there exist disjoint subsets $T$, $M$, and $B$ of $Y$ such that:
    \begin{enumerate}
        \item $\pi_4(B,G) \geq \frac{\pi_4(Y,G)}{2}$.
        \item $N(T) \cap Y \subseteq M$.
        \item $G[T]$ and $G[B]$ are connected.
        \item Every vertex $u \in B$ has a neighbor in $M$ and every vertex $u \in M$ has a neighbor in $T$.
    \end{enumerate}
\end{lemma}
\begin{proof}
    Let $\ell := \pi_4(Y,G)$. Within this proof, for every $u, v\in Y$, we define the \emph{distance} between $u$ and $v$ as the length of the shortest $u$--$v$ path in the induced subgraph $G[Y]$. Now fix an arbitrary vertex $v_0 \in Y$, and for every integer $i \geq 0$, let $D_i$ denote the set of vertices in $Y$ whose distance from $v_0$ is exactly $i$. 
    \begin{claim}\label{claim:heavy_index}
        There exists an index $i$ such that $\pi_4(D_i,G) \geq \frac{\ell}{2}$.
    \end{claim}
    \begin{claimproof}
        Suppose for contradiction that $\pi_4(D_i,G) < \frac{\ell}{2}$ for every $i \geq 0$. Note that $D_i$ and $D_{i'}$ are anticomplete in $G$ whenever $|i-i'| \geq 2$. Indeed, any edge $(v,v')$ with $v \in D_i$, $v' \in D_{i'}$, and $i' \geq i+2$ yields a walk of length $i+1$ from $v_0$ to $v'$ in $G[Y]$, contradicting $v' \in D_{i'}$. Define $D_{\sf even} := \bigcup_{i \geq 0} D_{2i}$ and $D_{\sf odd} := \bigcup_{i \geq 0} D_{2i+1}$. Since any two even-indexed layers are anticomplete, the tightness condition of Lemma~\ref{lem:properties_of_pi} gives $\pi_4(D_{\sf even},G) = \max\{\pi_4(D_{2i},G)\mid i\geq 0\}$. The same argument applied to odd-indexed layers gives $\pi_4(D_{\sf odd},G) = \max\{\pi_4(D_{2i+1},G)\mid i \geq 0\}$. The upper bound of Lemma~\ref{lem:properties_of_pi} then yields $\pi_4(Y,G) \leq \pi_4(D_{\sf even},G) + \pi_4(D_{\sf odd},G) < \frac{\ell}{2} + \frac{\ell}{2} = \ell$, a contradiction.
    \end{claimproof}
    Let $i_o$ be such that $\pi_4(D_{i_o},G) \geq \frac{\ell}{2}$, and let $C$ be a component of $G[D_{i_o}]$ with $\pi_4(C, G) = \pi_4(D_{i_o},G)$. Indeed such a component must exist: since distinct connected components of $G[D_{i_o}]$ are anticomplete by definition, the tightness condition of Lemma~\ref{lem:properties_of_pi} implies $\pi_4(D_{i_o},G) = \max_C \pi_4(C,G)$, where the maximum is over all components $C$ of $G[D_{i_o}]$. Hence some component $C$ satisfies $\pi_4(C,G) = \pi_4(D_{i_o},G) \geq \frac{\ell}{2}$.

    We claim that the sets $T := \bigcup_{i \leq i_o - 2} D_i$, $M := D_{i_o - 1}$, and $B := C$ satisfy all required conditions. Condition (1) holds because $B = C$ was chosen so that $\pi_4(C, G) \geq \frac{\ell}{2}$. For (2), suppose $v \in Y \setminus T$ has a neighbor $u \in T$. Since $u$ lies at distance at most $i_o - 2$ from $v_0$, we infer that $v$ lies at distance at most $i_o - 1$, and hence $v \in M$, establishing $N(T) \cap Y \subseteq M$. Condition (3) follows from the fact that $G[B]$ is connected, as $B = C$ is a connected component of $G[D_{i_o}]$; and $G[T]$ is connected since every vertex in $T$ has a path to $v_0$ entirely within $T$. Finally, for (4), observe that by construction every vertex in $D_i$ (for $i \geq 1$) has a neighbor in $D_{i - 1}$ (its predecessor along a shortest $v_0$--$v$ path in $G[Y]$), ensuring the necessary adjacency between $T$, $M$, and $B$. This concludes the proof of the lemma. 
\end{proof}

Now, we are ready to prove Theorem~\ref{thm:ktt_free_implies_bounded_radius_partition}.

\KttFreeimpliesBoundedRadiusPartition*
\begin{proof}
    Suppose for contradiction that $\pi_4(G) > t \cdot 2^t$. We make the following claim.\\[-4pt]
    
    \noindent\textbf{Claim~9.2.1.}\label{claim:recursive_layering}
        \emph{There exists a sequence of vertex subsets $V_1 \supset V_2 \supset \dots \supset V_{t+1}$ in $G$ such that, for every $i \in [t]$, there exist disjoint subsets $T_i$ and $M_i$ of $V_i\setminus V_{i+1}$ such that:
        \begin{enumerate}
            \item $G[T_i]$ is connected.
            \item $N(T_i) \cap V_i \subseteq M_i$.
            \item Every vertex $u \in M_i$ has a neighbor in $T_i$.
            \item Every vertex $u \in V_{i+1}$ has a neighbor in $M_i$.
        \end{enumerate}
        Moreover, there exists a subset $U \subseteq V_{t+1}$ of cardinality $t$ such that for any distinct $u, v \in U$, there is no $u$--$v$ path of length at most four in $G$.}
    % Suppose for contradiction that $\pi_4(G) > t \cdot 2^t$. We make the following claim.
    
    % \begin{claim}\label{claim:recursive_layering}
    %     There exists a sequence of vertex subsets $V_1 \supset V_2 \supset \dots \supset V_{t+1}$ in $G$ such that, for every $i \in [t]$, there exist disjoint subsets $T_i$ and $M_i$ of $V_i\setminus V_{i+1}$ such that:
    %     \begin{enumerate}
    %         \item $G[T_i]$ is connected.
    %         \item $N(T_i) \cap V_i \subseteq M_i$.
    %         \item Every vertex $u \in M_i$ has a neighbor in $T_i$.
    %         \item Every vertex $u \in V_{i+1}$ has a neighbor in $M_i$.
    %     \end{enumerate}
    %     Moreover, there exists a subset $U \subseteq V_{t+1}$ of cardinality $t$ such that for any distinct $u, v \in U$, there is no $u$--$v$ path of length at most four in $G$.
    % \end{claim}
    \begin{claimproof}
        Let $V_1 = V(G)$. Starting with $i = 1$, we iteratively construct the sets $T_i$, $M_i$, and $B_i$ by applying Lemma~\ref{lem:obtain_new_layer} to $V_i$, and then define $V_{i+1} := B_i$. We repeat this process for each $i \in [t]$. This is valid because, by construction, each $V_{i+1}$ is connected and therefore satisfies the conditions of the lemma in the next iteration. Observe that, by construction, $V_i \supset V_{i+1}$ for every $i\in [t]$. Moreover, the four required properties of the sets $T_i$, $M_i$, and $V_{i+1}$ follow directly from the corresponding properties guaranteed by Lemma~\ref{lem:obtain_new_layer}. 

        Let $U = \{u_1,\dots,u_{t'}\}$ be an inclusion-maximal subset such that no two distinct vertices of $U$ are connected by a path of length at most four in $G$. We claim $t' \geq t$. For each $i \in [t']$, let $X_i$ be the set of vertices of $V_{t+1}$ that have a path of length at most four to $u_i$ in $G$. Then $\pi_4(X_i,G) = 1$, since $u_i$ witnesses $(1,4)$-coverability of $X_i$. Furthermore $(X_1,\dots,X_{t'})$ is a partition of $V_{t+1}$ by maximality of $U$. Thus the upper bound of Lemma~\ref{lem:properties_of_pi} gives $\pi_4(V_{t+1},G) \leq t'$. But since the sets $V_i$ satisfy the relation $\pi_4(V_{i+1}, G) \geq \frac{\pi_4(V_i, G)}{2}$, we conclude $\pi_4(V_{t+1}, G) \geq \frac{\pi_4(G)}{2^t} > t$, which completes the proof of the claim. 
        % 
        % It remains to prove the existence of a subset $U \subseteq V_{t+1}$ of cardinality $t$ such that no two distinct vertices in $U$ are connected by a path of length at most four in $G$\todo{a nicer way of sa}. To this end, let $U$ be any inclusion-maximal subset of $V_{t+1}$ with this property. We claim that $|U| \geq t$. Notice that for every $i \in [t]$, the sets $V_i$ satisfy the relation $\pi_4(V_{i+1}, G) \geq \frac{\pi_4(V_i, G)}{2}$, which implies that $\pi_4(V_{t+1}, G) \geq \frac{\pi_4(G)}{2^t} > t$. Suppose, for contradiction, that $|U| = t' < t$. We construct a four–radius vertex partition of $V_{t+1}$ in $G$ as follows. Order the elements of $U$ arbitrarily, and for each $i \in [t']$, starting from $i = 1$, let $X_i$ be the set of vertices in $V_{t+1} \setminus \bigcup_{j=1}^{i-1} X_j$ that have an induced path of length at most four to the $i$th element $u_i$ of $U$. By the maximality of $U$, every vertex in $V_{t+1}$ belongs to some $X_i$, and by construction, the sets in $\{X_i\}_{i\in [t]}$ are pairwise disjoint. Thus, $(X_1, \dots, X_{t'})$ forms a four–radius vertex partition of $V_{t+1}$ in $G$, contradicting the fact that $\pi_4(V_{t+1}, G) > t'$, thereby completing the proof.
    \end{claimproof}
    
    Let $V_i$, $T_i$ and $M_i$ for every $i \in [t]$, together with $V_{t+1}$ and $U = \{u_1,\dots, u_t\} \subseteq V_{t+1}$, denote the sets whose existence is guaranteed by Claim~\hyperref[claim:recursive_layering]{9.2.1}. We construct an induced-minor model of $K_{t,t}$ in $G$, where sets $\{A_i\}_{i \in [t]}$ and $\{B_i\}_{i \in [t]}$ form the two sides of the bipartite graph in the model. For each $i \in [t]$, we define $A_i := T_i$ and $B_i := \{u_i\}\, \cup\, \{v_i^j\}_{j \in [t]}$, where $v_i^j$ is a neighbor of $u_i$ in $M_j$. Such a neighbor always exists because $u_i \in V_j$ for every $j \in [t]$ -- this holds since $u_i \in V_{t+1}$ and, by construction, $V_{i'+1} \subseteq V_{i'}$ for all $i' \in [t]$ -- and, by the properties of the construction, every vertex in $V_{i+1}$ has at least one neighbor in $M_i$.

    Now we need to argue that this is an induced-minor-model of $K_{t,t}$. Observe that by definition, $A_i$ and $B_i$ for all $i\in [t]$ induce connected subgraphs of $G$. Let $v\in A_i$, $v'\in A_{i'}$ for some $i > i'$. Note that $v\in V_{i}$ and consequently in $V_{i'+1}$, while $v'\in T_{i'}$. Therefore, $v$ and $v'$ are distinct as $V_{i'+1}$ and $T_{i'}$ are disjoint, and non-adjacent as $N(T_{i'}) \cap V_{i'}$ is contained within $M_{i'}$ which is disjoint from $V_{i'+1}$. Similarly, let $v\in B_i$, $v'\in B_{i'}$ for some $i \neq i'$. Note that if $(u_i, v, v', u_{i'})$ is a walk in $G$, it would contradict the fact that there is no $u_i$--$u_{i'}$ path on at most four vertices. Hence $v$ and $v'$ must be distinct and non-adjacent. Finally, let $i, j \in [t]$. We claim that there exists a vertex in $A_i$ that is adjacent to $v_j^i \in B_j$. Indeed, by construction, $v_j^i$ lies in $M_i$, and by Claim~\hyperref[claim:recursive_layering]{9.2.1}, every vertex in $M_i$ has a neighbor in $T_i$, which is precisely $A_i$. This proves that $\{A_i, B_i\}_{i \in [t]}$ forms an induced-minor-model of $K_{t,t}$ in $G$, contradicting our assumption that $G$ is $K_{t,t}$-induced-minor-free, and thereby completing the proof.
\end{proof}

\section{Proof of Theorem~\ref{thm:coarse_treewidth_our_class}}
Towards the proof of Theorem~\ref{thm:coarse_treewidth_our_class}, we state the following two propositions: the first is due to Hajebi, and the second due to Gir\~{a}o and Hunter; the latter also follows from a result of Bourneuf, Buci\'{c}, Cook, and Davies. We remark that the first proposition is stated in a weaker form than the result of Hajebi; however, this formulation suffices for our purposes.

\begin{proposition}[Lemma 5.1~\cite{hajebi2025polynomialboundspathwidth}]\label{prop:our_class_deg_tw_bound}
    For every positive integer $t$, there exists a positive integer $c(t)$ such that any $\{K_{t,t}, \boxplus_t\}$-induced-minor-free graph $G$ satisfies $\tw(G)\leq \Delta(G)^{c(t)}$.
\end{proposition}

% \begin{proposition}[Theorem 1.1~\cite{Bourneuf2024On}]
%     For every graph H, there is a polynomial p, such that every Ks,s-subgraph-free graph with no induced subdivision of H has average degree at most p(s).
% \end{proposition}

\begin{proposition}[Theorem 1.5~\cite{10.1093/imrn/rnaf025}, Theorem 1.1~\cite{MatijaPolyDegen} ]\label{prop:kuhn_osthus}
    For every integer $k,t\geq 2$, there exists a positive integer $d(t,k)$ such that one of the following holds:
    \begin{itemize}
        \item $G$ contains $K_k$.
        \item $G$ contains $K_{t,t}$ as an induced subgraph.
        \item $G$ contains $K_k$ as a proper induced subdivision.
        \item $G$ has average degree at most $d(t,k)$.
    \end{itemize}
\end{proposition}

Since any graph that contains a large clique as a proper induced subdivision also contains a large $K_{t,t}$ as an induced minor, and since forbidding a graph as a subgraph or as an induced minor is hereditary, Proposition~\ref{prop:kuhn_osthus} directly implies the following.

\begin{proposition}\label{prop:bounded_degeneracy}
    For every integer $t\geq 2$, there exists a positive integer $d(t)$ such that every graph that is $K_t$-free and $K_{t,t}$-induced-minor-free has degeneracy at most $d(t)$
\end{proposition}

With this, we are ready to prove Theorem~\ref{thm:coarse_treewidth_our_class}.\\[-5pt]

% \noindent\textbf{Theorem 1.1.}\emph{ Let $t$ be a positive integer and let $G$ be a $K_{t,t}$-induced-minor-free graph. Then $G$ either has a $\boxplus_t$-induced-minor or it has a tree decomposition $(T,\chi)$ such that for every $x\in V(T)$, the largest bag $\chi(x)$ is $(k, r)$--coverable in $G$, with $k$ and $r$ satisfying:
% \begin{align*}
%     k\quad &\leq\quad 136000 \cdot \log^3 (2n) \cdot f \log (f+4), \\
%     r\quad &\leq\quad 8\log (2n).
% \end{align*}    
% where $f$ equals $(5500000\cdot d(2^{t\log t})\cdot \log^5 (2n))^{c(t)}$, and where $c,d$ are the functions described in Proposition~\ref{prop:our_class_deg_tw_bound} and~\ref{prop:bounded_degeneracy} respectively.}

\noindent\textbf{Theorem 1.1.}\emph{ Let $t$ be a positive integer and let $G$ be a $K_{t,t}$-induced-minor-free graph. Then $G$ either has $\, \boxplus_t$ as an induced minor or has a tree decomposition $(T,\chi)$ such that for every $x\in V(T)$, the distance $16\log (4n)$-independence number of the bag $\chi(x)$ is at most
$$
115000 \cdot \log^3 (4n) \cdot f \log (4f);\quad 
\text{ where }
f = (1000000\cdot d(t\cdot 2^t)\cdot \log^5 (4n))^{c(t)}.
$$
where $c,d$ are the functions described in Propositions~\ref{prop:our_class_deg_tw_bound} and~\ref{prop:bounded_degeneracy}, respectively.}

\begin{proof}
    Assume that $G$ has no $\boxplus_t$-induced minor, as otherwise the statement holds. Consider a four-radius vertex partition $(X_1,\dots,X_{\pi_4(G)})$ of $G$ with the minimum possible number of parts. We define $G'$ to be the induced minor of $G$ obtained by contracting, for each $i \in [\pi_4(G)]$ and every connected component $C$ of $G[X_i]$, the edges of $G[C]$. Note that the vertex set of $G'$ is $V(G') = \{ v_{C} \mid C \text{ is a connected component of } G[X_i] \text{ for some } i \in [\pi_4(G)]\ \}$, and two vertices $v_{C_1}$ and $v_{C_2}$ are adjacent in $G'$ whenever there exist vertices $u_1 \in C_1$ and $u_2 \in C_2$ with $(u_1,u_2) \in E(G)$.

    Since $G'$ is obtained by contracting edges of a $\{K_{t,t}, \boxplus_t\}$-induced-minor-free graph, it is itself $\{K_{t,t}, \boxplus_t\}$-induced-minor-free. Therefore, by Proposition~\ref{prop:our_class_deg_tw_bound}, we have $\tw(G') \leq \Delta(G')^{c(t)}$. Furthermore, for every $i\in[\pi_4(G)]$, the set $\{ v_{C}\mid C \text{ is a connected component of } G[X_i]\}$ is an independent set in $G'$. Hence the chromatic number of $G'$, and consequently the size of its largest clique, is at most $\pi_4(G) \leq t \cdot 2^{t}$. By Proposition~\ref{prop:bounded_degeneracy}, it follows that $G'$ has degeneracy at most $d \leq d(t\cdot 2^{t})$. Consequently, $G' \in \mathcal{C}^c_d$ for $c := c(t)$ and $d := d(t\cdot 2^{t})$.
    Applying Corollary~\ref{thm:coarse_treewidth_aux_class} to $G'$ we obtain a tree decomposition $(T',\chi')$ such that for every $x\in V(T')$, the bag $\chi'(x)$ is $(k', r')$--coverable in $G'$, for $k'\, \leq\, 115000 \cdot \log^3 (4n) \cdot f \log (4f)$, and $r'\, \leq\, \log (4n)$ where $f$ equals $(1000000d\cdot \log^5 (4n))^{c}$.
    
    We define a pair $(T,\chi)$ as follows. Let $T := T'$, and for every node $x \in V(T)$ define $\chi(x) := \{\, u \in V(G) \mid \text{$u \in C \subseteq X_i$ and $v_{C} \in \chi'(x)$} \,\}$. A direct application of definitions suffices to verify that $(T,\chi)$ is indeed a tree decomposition of $G$. 
    Now, consider $x\in V(T)$ and let $I$ be a largest $16\log (4n)$-independent set of $G$ contained in $\chi(x)$. Let $I' := \{v_C\in V(G') \mid C\cap I \neq \emptyset\}$. Observe that $|I'| = |I|$. Indeed, every vertex of $I$ belongs to some connected component $C$ of $G[X_i]$ for some $i\in[\pi_4(G)]$. Moreover, no two vertices $u,v \in I$ belong to the same connected component $C$ of $G[X_i]$, since each such component is $(1,4)$-coverable and hence contains a $u$--$v$ path on at most eight vertices, contradicting the $16\log (4n)$-independence of $I$.
    Let $\mathcal{C}' \subseteq V(G')$ be a set witnessing that $\chi'(x)$ is $(k',r')$-coverable, i.e., $|\mathcal{C}'| \leq 115000 \cdot \log^3 (4n) \cdot f \log (4f)$ and for every vertex $v_C$ of $\chi'(x)$, there exists at least one vertex in $\mathcal{C}'$ that has a path on at most $\log (4n)$ vertices to $v_C$. If $|I'| > 115000 \cdot \log^3 (4n) \cdot f \log (4f)$, then by the pigeonhole principle there exists a vertex of $\mathcal{C}'$ that has a path on at most $\log (4n)$ vertices to two distinct vertices $v_{C_1}, v_{C_2} \in I'$.
    This implies the existence of a path on at most $2\log (4n)$ vertices between $v_{C_1}$ and $v_{C_2}$ in $G'$. Since each component $C$ is $(1,4)$-coverable and therefore contains a path on at most eight vertices between any two of its vertices in $G$, we obtain a path on at most $16\log (4n)$ vertices between $u\in C_1\cap I$ and $u'\in C_2\cap I$, contradicting the independence of $I$ and completing the proof of the theorem.
\end{proof}

\section{Proof of Theorem~\ref{thm:coarse_subpoly_treewidth_our_class}}

For the proof of Theorem~\ref{thm:coarse_subpoly_treewidth_our_class}, we need the following theorem of Chudnovsky, Codsi, Fischer and Lokshtanov.

\begin{proposition}[Theorem 1.2~\cite{subpolynomialtreewidth}]\label{prop:subpoly_treewidth}
    Let $t$ be a positive integer and let $G$ be a graph that excludes $K_{t,t}$ and $\boxplus_t$ as induced minors and $K_t$ as a subgraph. Then there exists a positive integer $c'(t)$ and $\epsilon'(t)\in(0,1]$ such that $G$ has a tree decomposition of width at most $2^{c'(t) \log^{1-\epsilon'(t)} n}$.
\end{proposition}

Observe that Theorem~\ref{thm:coarse_subpoly_treewidth_our_class} follows immediately from Proposition~\ref{prop:subpoly_treewidth} and Theorem~\ref{thm:ktt_free_implies_bounded_radius_partition}.

\CoarseSubpolyTreewidthOurClass*
\begin{proof}
    Consider a four-radius vertex partition $(X_1,\dots,X_{\pi_4(G)})$ of $G$ with the minimum possible number of parts. We define $G'$ to be the induced minor of $G$ obtained by contracting, for each $i \in [\pi_4(G)]$ and every connected component $C$ of $G[X_i]$, the edges of $G[C]$. In particular, the vertex set of $G'$ is $V(G') = \{ v_{C} \mid C \text{ is a connected component of } G[X_i] \text{ for some } i \in [\pi_4(G)]\ \}$, and two vertices $v_{C_1}$ and $v_{C_2}$ are adjacent in $G'$ whenever there exist vertices $u_1 \in C_1$ and $u_2 \in C_2$ with $(u_1,u_2) \in E(G)$. 
    
    Observe that for every $i\in[\pi_4(G)]$, the set $\{ v_{C}\mid C \text{ is a connected component of } G[X_i]\}$ is an independent set in $G'$. Hence by Theorem~\ref{thm:ktt_free_implies_bounded_radius_partition}, the chromatic number of $G'$, and consequently the size of its largest clique, is at most $\pi_4(G) \leq t \cdot 2^{t}$. Therefore, by Proposition~\ref{prop:subpoly_treewidth}, it follows that $G'$ has a tree decomposition $(T',\chi')$ such that for every $x\in V(T')$, $|\chi'(x)|\leq 2^{c(t) \log^{1-\epsilon(t)} n}$, where $c(t) := c'(t \cdot 2^{t})$ and $\epsilon(t) := \epsilon'(t \cdot 2^{t})$. 
    
    We define a pair $(T,\chi)$ as follows. Let $T := T'$, and for every node $x \in V(T)$ define $\chi(x) := \{\, u \in V(G) \mid \text{$u \in C \subseteq X_i$ and $v_{C} \in \chi'(x)$} \,\}$. A direct application of definitions suffices to verify that $(T,\chi)$ is indeed a tree decomposition of $G$.
    Now assume that there exists a node $x \in V(T)$ such that the distance-$8$ independence number of $\chi(x)$ is greater than $2^{c(t)\log^{1-\epsilon(t)} n}$. Let $I \subseteq \chi(x)$ be a distance-$8$ independent set of maximum size. By the pigeonhole principle, there exist distinct vertices $u,v \in I$ and a vertex $v_C \in \chi'(x)$ such that $u,v \in C$, where $C$ denotes the component corresponding to $v_C$. Since $C$ is $(1,4)$--coverable, there exists a path between $u$ and $v$ on at most eight vertices. This contradicts the assumption that $I$ is distance $8$-independent, and therefore completes the proof of the theorem.
\end{proof}

\section{Proof of  Theorem~\ref{thm:coarse_induced_menger_our_class}}
For the proof of Theorem~\ref{thm:coarse_induced_menger_our_class} we require the following three propositions: the first follows directly from a theorem of Bonnet, Hodor, Korhonen and Masa{\v{r}}{\'\i}k, the second is the classical theorem of Menger, and the third is due to Gartland, Korhonen, and Lokshtanov.

\begin{proposition}[Theorem 11~\cite{bonnet2023treewidth}]\label{prop:our_class_initial_edge_coloring}
     For every positive integer $t$, there exists a positive integer $\mu(t)$ such that any $K_{t,t}$-induced-minor-free graph $G$ has a partition $\Lambda := (\Lambda_1,\dots, \Lambda_{\mu(t)})$ of $E(G)$ where, for every $i\in [\mu(t)]$, every connected component of the graph $(V(G), \Lambda_i)$ induces a star in $G$.
\end{proposition}

\begin{proposition}[Menger's Theorem~\cite{Menger}]\label{prop:mengers_theorem}
    Let $G$ be a graph and $A,B \subseteq V(G)$. Then for every positive integer $k$, $G$ either has a collection of $k$ vertex-disjoint $A$--$B$ paths, or an $A$--$B$ separator of size at most $k$.
\end{proposition}

\begin{proposition}[Theorem 1~\cite{gartland2023inducedversionsmengerstheorem}]\label{prop:sparse_induced_menger}
    Let $G$ be a graph of maximum degree $\Delta$ and $A, B \subseteq V (G)$. Then for every positive integer $k$, there either exists a collection of $k$ vertex-disjoint and pairwise anticomplete $A$--$B$ paths, or there exists an $A$--$B$ separator of size at most $k \cdot (\Delta + 1)^{\Delta^2+1}$.
\end{proposition}

Given a graph $G$ and a positive integer $s$, an \emph{$s$-size edge partition} is defined to be a partition $\Lambda := (\Lambda_1,\dots, \Lambda_{\ell})$ of the edge set of $G$ such that every connected component of the graph $(V(G), \Lambda_i)$ has at most $s$ vertices. Proposition~\ref{prop:our_class_initial_edge_coloring} directly yields the following.

\begin{proposition}\label{prop:our_class_edge_coloring}
     For every positive integer $t$, there exists a positive integer $\mu(t)$ such that any $K_{t,t}$-induced-minor-free graph $G$ of maximum degree $\Delta$ has a $(\Delta+1)$-size edge partition $\Lambda$ with at most $\mu(t)$ parts.
\end{proposition}

Let $G$ be a graph and $\Lambda := (\Lambda_1,\dots,\Lambda_{\ell})$ be a partition of the edge set of $G$. For $i\in [\ell]$, define the part $\Lambda_i$ to be \emph{clean} if maximum degree of the graph $(V(G), \Lambda_i)$ is two, and \emph{cluttered} otherwise.

\begin{lemma}\label{lem:induced_menger_cleaning_step}
    Let $G$ be a graph and $\Lambda$ be an $s$-size edge partition with $\ell$ parts, of which at least one is cluttered. Let $A,B \subseteq V(G)$ and let $f$ be a positive integer such that $G$ has no $A$--$B$ separator of size at most $f$. Then there exists an induced subgraph $H$ of $G$ with an $s$-size edge partition $\Lambda^H$ with $\ell$ parts and strictly more clean parts than $\Lambda$, and no $A_H$--$B_H$ separator of size at most $f/s$, where $A_H := A \cap V(H)$ and $B_H := B \cap V(H)$.
\end{lemma}
\begin{proof}
    Let $\Lambda_{i_o}$ be a part in $\Lambda$ which is cluttered. Let $G'$ be the minor of $G$ obtained by contracting the edges in $\Lambda_{i_o}$. Equivalently, $V(G')=\{v_C \mid C \text{ is a connected component of } (V(G),\Lambda_{i_o})\}$, and two vertices $v_{C_1}$ and $v_{C_2}$ are adjacent in $G'$ whenever there exist vertices $u_1 \in C_1$ and $u_2 \in C_2$ with $(u_1,u_2) \in E(G)$. Define $A' := \{v_C \in V(G') \mid C \cap A \neq \emptyset\}$ and $B' := \{v_C \in V(G') \mid C \cap B \neq \emptyset\}$.

    \begin{claim}
        $G'$ has no $A'$--$B'$ separator of size at most $f/s$.
    \end{claim}
    \begin{claimproof}
        Suppose otherwise that $S'$ is such a separator, and define $S := \{v \in V(G) \mid v \in C \text{ for some } v_C \in S'\}$. As $\Lambda$ is an $s$-size edge partition, each such component $C$ has size at most $s$, and consequently $|S| \le s\cdot |S'| \le f$. Therefore $S$ cannot be an $A$--$B$ separator, by the definition of $G$, $A$ and $B$, which implies the existence of an $A$--$B$ path $P$ in $G - S$. 
        We obtain a walk $P'$ in $G'-S'$ by replacing each vertex $v \in V(P)$ with the vertex $v_C$, where $C$ is the connected component of $(V(G),\Lambda_{i_o})$ containing $v$. By construction of $G'$, consecutive vertices of $P$ map to vertices in $G'$ that are the same or adjacent, and since $P$ avoids $S$, the walk $P'$ avoids $S'$. Moreover, as $P$ starts in $A$ and ends in $B$, $P'$ starts in $A'$ and ends in $B'$. But this contradicts the assumption that $S'$ is an $A'$--$B'$ separator, thereby completing the proof of the claim. 
    \end{claimproof}

    Therefore, by Menger's theorem~\ref{prop:mengers_theorem}, there exists a collection $\p'$ of vertex-disjoint $A'$--$B'$ paths in $G'$ with $|\p'| \ge f/s$. We construct from $\p'$ a collection $\p$ of $A$--$B$ paths in $G$ as follows. For every path $P' \in \p'$, let $V_{P'} := \{v \in V(G) \mid v \in C \text{ for some } v_C \in V(P')\}$. Since the first vertex of $P'$ lies in $A'$ and the last lies in $B'$, the set $V_{P'}$ intersects both $A$ and $B$. Hence there exists an induced $A$--$B$ path $P$ in $G[V_{P'}]$, which we add to $\p$. Finally, since the paths in $\p'$ are vertex-disjoint, the sets $V_{P'}$ are pairwise disjoint, and thus the corresponding paths in $\p$ are vertex-disjoint. Consequently, $\p$ is a collection of vertex-disjoint $A$--$B$ paths in $G$ with $|\p| \ge f/s$.

    Let $V_{\p} := \{v \in V(G) \mid v \in P \text{ for some } P \in \p\}$ and let $H := G[V_{\p}]$. For each $i \in [\ell]$, define $\Lambda^H_i := \{(u,v) \in \Lambda_i \mid u,v \in V_{\p}\}$. We claim that $\Lambda^H := (\Lambda^H_1,\dots,\Lambda^H_\ell)$ satisfies the requirements of the lemma.
    By definition, $\Lambda^H$ is a partition of $E(H)$ into $\ell$ parts. Since $\Lambda^H_i \subseteq \Lambda_i$ and $V(H) \subseteq V(G)$, every connected component of $(V(H),\Lambda^H_i)$ is contained in a component of $(V(G),\Lambda_i)$. Hence $\Lambda^H$ is an $s$-size edge partition.
     
    Now fix $i \in [\ell]$ such that $\Lambda_i$ is clean. If a vertex $v \in V(H)$ has at least three neighbors in $(V(H),\Lambda^H_i)$, then it has the same neighbors in $(V(G),\Lambda_i)$, contradicting the definition of $i$. Hence $\Lambda^H_i$ is clean.
    Finally, consider the graph $(V(H),\Lambda^H_{i_o})$ and suppose it contains a vertex $v$ of degree greater than two. Let $P$ be the path in $\p$ containing $v$. Since $P$ is induced, at most two neighbors of $v$ are in $P$ and hence at least one neighbor of $v$ in $(V(H),\Lambda^H_{i_o})$ must belong to a different path $Q$. Thus both $P$ and $Q$ intersect the connected component $C$ of $(V(G),\Lambda_{i_o})$ that contains $v$. But then, $P'$ and $Q'$ must both contain $v_C$, which contradicts the fact that they were vertex-disjoint. Therefore $\Lambda^H_{i_o}$ is clean, completing the proof.
\end{proof}

\begin{lemma}\label{lem:induced_menger_recursive_step}
    Let $G$ be a graph with an $s$-size edge partition $\Lambda$ with $\ell$ parts, of which $z$ are cluttered, for positive integers $s,z,\ell$. Let $k$ be a positive integer and let $A,B \subseteq V(G)$. Then $G$ either has a collection of $k$ vertex-disjoint and pairwise anticomplete $A$--$B$ paths, or has an $A$--$B$ separator of size at most $k \cdot g(s,z,\ell)$, where $g(s,z,\ell) := s^{z} \cdot (2\ell+1)^{4\ell^2+1}$.
\end{lemma}
\begin{proof}
    We prove the lemma by induction on $z$. For the base case, observe that when $z$ is zero, every part in $\Lambda$ is clean and consequently the maximum degree of $G$ is at most $2\ell$. Applying Proposition~\ref{prop:sparse_induced_menger} to $G$ asserts the lemma for $z = 0$.
    Fix $z_o \ge 1$ and assume that the lemma holds for all instances with $z < z_o$. If $G$ has an $A$--$B$ separator of size at most $k \cdot g(s,z_o,\ell)$, then the lemma holds and we are done. Otherwise, we apply Lemma~\ref{lem:induced_menger_cleaning_step} to $(G,\Lambda,A,B)$ and obtain an induced subgraph $H$ together with an $s$-size edge partition $\Lambda^H$ with $\ell$ parts, of which fewer than $z_o$ are cluttered. Let $A_H := A \cap V(H)$ and $B_H := B \cap V(H)$, and apply the inductive hypothesis to the instance $(H,\Lambda^H,k,A_H,B_H)$. Since $H$ admits no $A_H$--$B_H$ separator of size at most $\frac{k\cdot g(s,z_o,\ell)}{s} \le k\cdot g(s,z_o-1,\ell)$, we conclude that $H$ contains a collection of $k$ vertex-disjoint and pairwise anticomplete paths. As $H$ is an induced subgraph of $G$, these paths remain vertex-disjoint and pairwise anticomplete in $G$, which concludes the proof.
\end{proof}

Proposition~\ref{prop:our_class_edge_coloring}, combined with the fact that the number of cluttered parts is at most the number of parts, immediately yields the following corollary.

\begin{corollary}\label{thm:induced_menger_degree_dependent_our_class}
    Let $\Delta, t$ be positive integers and let $G$ be a $K_{t,t}$-induced-minor-free graph of maximum degree at most $\Delta$. Let $k$ be a positive integer and let $A,B \subseteq V(G)$. Then $G$ either has a collection of $k$ vertex-disjoint and pairwise anticomplete $A$--$B$ paths, or an $A$--$B$ separator of size at most $k \cdot g(\Delta,t)$, where $g(\Delta,t) := (\Delta+1)^{\mu} \cdot (2\mu+1)^{4\mu^2+1}$ and $\mu$ is the constant defined in Proposition~\ref{prop:our_class_edge_coloring}.
\end{corollary}

We now proceed to the proof of \Cref{thm:coarse_induced_menger_our_class}, which we restate for convenience.\\[-5pt]

\noindent\textbf{\Cref{thm:coarse_induced_menger_our_class}.}\emph{ Let $t$ be a positive integer and let $G$ be a $K_{t,t}$-induced-minor-free graph. Let $A,B \subseteq$\\ $V(G)$ and $k$ be a positive integer. Then $G$ either has a collection of $k$ vertex-disjoint and pairwise anticomplete $A$--$B$ paths or an $A$--$B$ separator $S\subseteq V(G)$ whose distance $16\log (4n)$-independence number is $100 k \cdot g \cdot \log^3 (4n)$, where $g := ( 24 \log^2 (4n) + 4d + 1)^{\mu} \cdot (2\mu +1)^{4\mu^2+1}$ and $d, \mu$ are the constants defined in Proposition~\ref{prop:bounded_degeneracy} and~\ref{prop:our_class_edge_coloring}, respectively. }

\begin{proof}
    Consider a four-radius vertex partition $(X_1,\dots,X_{\pi_4(G)})$ of $G$ with at most $t\cdot 2^t$ parts, whose existence is asserted by Theorem~\ref{thm:ktt_free_implies_bounded_radius_partition}. We define $G'$ to be the induced minor of $G$ obtained by contracting, for each $i \in [\pi_4(G)]$, all edges inside the connected components of $G[X_i]$. Explicitly, the vertex set of $G'$ is $V(G') = \{ v_{C} \mid C \text{ is a connected component of } G[X_i], \text{ for some }i \in [\pi_4(G)] \}$, and two vertices $v_{C_1}$ and $v_{C_2}$ are adjacent in $G'$ whenever there exist vertices $u_1 \in C_1$ and $u_2 \in C_2$ with $(u_1,u_2) \in E(G)$. We use $n'$ to denote the number of vertices in $G'$.

    Since $G'$ is obtained from a $K_{t,t}$-induced-minor-free graph by contracting edges, it is itself $K_{t,t}$-induced-minor-free. Also, the set $\{\,v_{C}\mid C \text{ is a connected component of } G[X_i]\,\}$ is an independent set in $G'$, for every $i \in [\pi_4(G)]$. So it follows that the chromatic number, and hence the clique number, of $G'$ is at most $\pi_4(G) \leq t \cdot 2^{t}$. Hence, by Proposition~\ref{prop:bounded_degeneracy}, $G'$ has degeneracy at most $d \leq d(t\cdot 2^t)$.
    Define $A' := \{v_{C} \in V(G') \mid C \cap A \neq \emptyset\}$ and $B' := \{v_{C} \in V(G') \mid C \cap B \neq \emptyset\}$. Let $\f$ be the $4d$-neighborhood-witnessing layered family of $G'$ with thickness at most $\log (4n')$, whose existence is guaranteed by Lemma~\ref{lem:degeneracy_implies_layered_family}. Applying Theorem~\ref{thm:coarse_induced_menger_aux_class} to $(G',A',B',\f)$ with $f := 12k\cdot g + 1$ yields either a vertex set $S'$ or an induced subgraph $H'$ of $G'$.  
    
    % $\mathrm{fcov}_{\f}(S)\, \leq\, 8 \log^2 (4n) \cdot f$. 
    % \item An induced subgraph $H$ of $G$, of maximum degree at most $24 \log^2 (4n) + 4d$, such that every $A_H$--$B_H$ separator $S$ in $H$ must have $\mathrm{cov}_{\f}(S)\, \geq\, \frac{f}{12}$ in $G$, where $A_H := A\cap V(H)$ and $B_H := B\cap V(H)$.

    Suppose we obtain an induced subgraph $H'$ of maximum degree at most $24 \log^2 (4n') + 4d$ such that every $A_{H'}$--$B_{H'}$ separator $S$ in $H'$ satisfies $\mathrm{cov}_{\f}(S) > k \cdot g$, where $A_{H'} := A' \cap V(H')$ and $B_{H'} := B' \cap V(H')$. Applying Corollary~\ref{thm:induced_menger_degree_dependent_our_class} to $(H', A_{H'}, B_{H'}, k)$ yields either a collection of $k$ vertex-disjoint and pairwise anticomplete $A_{H'}$--$B_{H'}$ paths, or an $A_{H'}$--$B_{H'}$ separator of size at most $k \cdot g$. The latter contradicts the definition of $H'$, and therefore $H'$ contains a collection $\p'$ of $k$ vertex-disjoint and pairwise anticomplete $A_{H'}$--$B_{H'}$ paths.
    We now lift $\p'$ to a collection $\p$ of $A$--$B$ paths in $G$. For each path $P' \in \p'$, let $V_{P'} := \{v \in V(G) \mid v \in C \text{ for some } v_{C} \in V(P')\}$. Since $P'$ starts in $A_{H'}$ and ends in $B_{H'}$, the set $V_{P'}$ intersects both $A$ and $B$, and hence contains an induced $A$--$B$ path $P$, which we add to $\p$. As the paths in $\p'$ are vertex-disjoint and pairwise anticomplete, the sets $V_{P'}$ are likewise disjoint and anticomplete, and so the corresponding paths in $\p$ also have these properties. Consequently, $\p$ is a collection of $k$ vertex-disjoint and pairwise anticomplete $A$--$B$ paths in $G$.

    Otherwise we obtain an $A'$--$B'$ separator $S'\subseteq V(G')$ in $G'$, such that $\mathrm{fcov}_{\f}(S')\, \leq\, 8 \log^2 (4n') \cdot (12k\cdot g + 1)$. By Proposition~\ref{prop:packing_covering_LP}, we conclude that $\mathrm{cov}_{\f}(S')\, \leq\, 100 k \cdot g \cdot \log^3 (4n')$. Let $\f'\subset \f$ be a cover of $S'$ of size at most $100 k \cdot g \cdot \log^3 (4n')$, and redefine $S' := \bigcup_{F\in\f'} F$. Let $S := \{v\in V(G) \mid v\in C,\text{ for some } v_{C}\in S'\}$. We claim that $S$ is an $A$--$B$ separator with the desired properties.
    First, observe that if it is not an $A$--$B$ separator, then $G-S$ must contain an $A$--$B$ path $P$. We obtain a walk $P'$ in $G'-S'$ by replacing each vertex $v \in V(P)$ with the vertex $v_{C}$, where for some $i\in[\pi_4(G)]$, $C$ is the connected component of $G[X_i]$ containing $v$. By construction of $G'$, consecutive vertices of $P$ map to vertices in $G'$ that are the same or adjacent, and since $P$ avoids $S$, the walk $P'$ avoids $S'$. Moreover, as $P$ starts in $A$ and ends in $B$, $P'$ starts in $A'$ and ends in $B'$. But this contradicts the assumption that $S'$ is an $A'$--$B'$ separator in $G'$. 
    Second, observe that since every connected component $C$ of $G[X_i]$ for each $i \in [\pi_4(G)]$ is $(1,4)$--coverable in $G$, for every pair of vertices $u,v \in C$ there exists a $u$--$v$ path on at most eight vertices in $G$. Furthermore, letting $\mathcal{C}_{\f'}$ denote the set of centers of the elements of $\f'$, every vertex $v_C \in S'$ has a path on at most $\log (4n')$ vertices in $G'$ to some vertex of $\mathcal{C}_{\f'}$. Let $\mathcal{C}\subseteq V(G)$ be a subset that contains one arbitrary vertex of $C$ for every $v_C\in\mathcal{C}_{\f'}$. Combining the above two observations, we conclude that every vertex of $S$ has a path on at most $8\log (4n') \leq 8\log (4n)$ vertices to some vertex of $\mathcal{C}$. Therefore, the maximum size of a distance $16\log (4n)$-independent set in $S$ is $|\mathcal{C}| \leq 100 k \cdot g \cdot \log^3 (4n)$, which completes the proof of the theorem. 
\end{proof} 

% -----------------------------------------------------------------------------------------------------------------

\bibliographystyle{alpha}
\bibliography{ref} 

\end{document}